\documentclass[a4paper,twoside,11pt]{article}
\usepackage{a4wide}
\usepackage[utf8]{inputenc}
\usepackage[english]{babel}
\usepackage{color} 
\usepackage{amsmath,amscd,amssymb,amsthm,amsfonts}
\usepackage{xparse}
\usepackage{mathabx}
\usepackage{mathrsfs}
\usepackage{bbm}
\usepackage{graphicx}
\usepackage{hyperref}
\usepackage{enumitem}
\providecommand{\keywords}[1]
{
  \small	
  \textbf{\textbf{Keywords.}} #1
} 
\newcommand{\reff}[1]{(\ref{#1})}
\newcommand{\ddr}{\mathrm{d}}

\theoremstyle{plain}
\newtheorem{theo}{Theorem}[section]
\newtheorem{theostar}{Theorem}

\newtheorem{cor}[theo]{Corollary}
\newtheorem{prop}[theo]{Proposition}
\newtheorem{proposition}[theo]{Proposition}
\newtheorem{lem}[theo]{Lemma}

\theoremstyle{remark}
\newtheorem{rem}[theo]{Remark}

\def\P{\ensuremath{\mathbb{P}}}

\def\R{\ensuremath{\mathbb{R}}}

\def\Z{\ensuremath{\mathbb{Z}}}

\def\to{\rightarrow}

\newcommand{\ind}{{\bf 1}}

\newcommand{\expp}[1]{\mathop {e^{ #1}}}

\usepackage{color}

\usepackage{import}

\title{Decompositions of CSBPs via Poissonian Intertwining}

\date{\today}

\author{Cl\' ement Foucart\footnote{CMAP, Ecole Polytechnique. Email: clement.foucart@polytechnique.edu} \ and Olivier H\' enard\footnote{LMO, Universit\'e Paris-Saclay. Email: olivier.henard@universite-paris-saclay.fr}}

\begin{document}
\maketitle

\begin{abstract}
We revisit certain decompositions of continuous-state branching processes (CSBPs), commonly referred to as skeletal decompositions,
through the lens of intertwining of semi-groups.
Precisely, we associate to a CSBP \(X\) with branching mechanism \(\psi\) a family of $\mathbb{R}_+\times \mathbb{Z}_+$-valued  branching processes \((X^\lambda, L^\lambda)\), indexed by a parameter \(\lambda \in (0, \infty)\), that satisfies an intertwining relationship with \(X\) through the Poisson kernel with parameter \(\lambda x\). The continuous component \(X^\lambda\) has the same law as $X$, while the discrete component $L^\lambda$, conditionally on \(X^\lambda_t\), has a Poisson distribution with parameter $\lambda X^\lambda_t$. 
The law of \((X^\lambda, L^\lambda)\) depends on the position of \(\lambda\) within \([0, \infty) = [0, \rho) \cup [\rho, \infty)\), where \(\rho\) is the largest positive root of \(\psi\).  
When \(\lambda \geq \rho\), various well-known results concerning skeleton decompositions are recovered. In the supercritical case (\(\rho > 0\)), when $\lambda<\rho$, a novel phenomenon arises: a birth term appears in the skeleton, corresponding to a one-unit proportional immigration from the continuous to the discrete component. Along the way, the class of continuous-time branching processes taking values in \(\mathbb{R}_+ \times \mathbb{Z}_+\) is constructed.
\end{abstract}

\section{Introduction}
This paper investigates how Continuous-State Branching Processes (CSBPs) can be decomposed into $\mathbb{R}_+\times \mathbb{Z}_+$-valued branching processes using intertwining relationships between Markovian semigroups. An intertwining relationship is a property of commutation of semigroups, or generators, through some kernel. Formally, two semigroups $P_t$ and $Q_t$ are said to be intertwined with respect to a kernel $\Lambda$ if they satisfy the following relationship
\begin{equation}\label{intro:intertwining}
\Lambda P_t=Q_t\Lambda, \quad \forall t\geq 0.\end{equation}

Intertwining is intimately connected to the problem of determining whether a function of a Markov process remains Markovian, as explored by Pitman and Rogers \cite{RP81}. 
It also arises in Markovian filtering, see e.g. Kurtz and Ocone \cite{zbMATH04069937}, Kurtz \cite{zbMATH01215969}, and Kurtz and Nappo \cite{zbMATH05919853}, where  Markov mapping theorems are established in the setting of martingale problems.  Numerous examples of intertwined Markov semigroups have been discovered, including those presented by Carmona et al. \cite{CPY98} and Pal and Shkolnikov \cite{arXiv:1306.0857}.

The intertwining theory has been further developed in various directions, becoming a fundamental tool in  stochastic processes. They play for instance an important role in the study of strong stationary times, see for instance Diaconis and Fill \cite{zbMATH04192817},  Miclo \cite{zbMATH07206399} and Arnaudon et al. \cite{zbMATH07811897}.  Intertwining likewise appears in analyzing certain interacting particle systems, see e.g. Floreani et al. \cite{zbMATH07854756}.  It also plays a key role in the so-called \textit{lookdown} construction of the genealogy of branching populations, see for instance Kurtz \cite{zbMATH07163705} and Etheridge and Kurtz \cite{zbMATH07114706} for recent works in this direction. 

The concept of decomposition of branching processes  into ``skeletons"  dates back to Harris \cite{Harris}, who showed that any supercritical Bienaymé-Galton-Watson process with positive probability of extinction can be embedded in a two-type branching process,  see also Athreya-Ney's book \cite[Chapter I-Section 12]{AthreyaNey}. In this framework, one type represents individuals with infinite line of descent (called prolific or immortal individuals), while the other corresponds to mortal individuals. This decomposition has become a key tool in branching process theory. In the context of Continuous State Branching Processes (CSBPs), the prolific individuals form a discrete branching process, while the non-prolific ones form a continuum, as explored by Bertoin et al. \cite{BF08}, see also  Lambert and Uribe Bravo, \cite{zbMATH07940610}, for a recent work in this direction in the context of splitting trees. The framework with spatial motion  has also garnered considerable attention, we refer for instance to Eckhoff et al. \cite{zbMATH06502682}, Fekete et al. \cite{zbMATH07284534} and the references therein.

In a seminal paper, Duquesne and Winkel \cite{DW07} introduced a nested family of discrete trees with edge lengths that is consistent under Poissonian sampling of the leaves; they proved that any such family embeds a Lévy real tree encoding the genealogy of a CSBP. In a similar spirit, Abraham and Delmas \cite{zbMATH06047803} have constructed a tree-valued Markov process, where evolving skeletons serve as the hprimary objects. See also  Abraham and Delmas \cite{zbMATH05598047} and Abraham et al. \cite{zbMATH05946939} for closely related studies. More recently, these decompositions have been revisited using Poissonization techniques applied to the stochastic differential equation with jumps solved by a CSBP, as discussed by Fekete et al. \cite{FFK19}. In all these works, Esscher transforms of the branching mechanism play a central role. To summarize, given an eventually positive branching mechanism $\psi$, and denoting by $\rho$ its largest positive root, a skeleton decomposition consists in ``grafting" subcritical continuous masses — evolving with a branching mechanism given by an Esscher transform of $\psi$ at the right of $\rho$ — onto a discrete branching process, so that the total continuous mass is a CSBP governed by $\psi$.

Although not always highlighted in the works previously cited -- see however \cite[Page 722]{BF08} and \cite[Page 1283]{zbMATH07940610} -- a common feature in such skeletal decompositions lies on the fact that the joint process encoding both the discrete component and the continuous mass satisfies the branching property. Two-type branching processes taking values in $\mathbb{R}_+\times \mathbb{Z}_+$ come therefore naturally into play. To our knowledge, no general treatment of this class of processes has been presented before. We summarize their fundamental properties in Section \ref{sec:preliminaries}, see Theorem \ref{thm:branchingbitype} and Proposition~\ref{Fellerprop}, and recall well-known facts about one-dimensional branching processes. The $\mathbb{R}_+\times \mathbb{Z}_+$-valued branching processes are constructed in Section \ref{sec:constructionbitype} by using classical results from the theory of Markov processes.

Our main results are presented in Section \ref{sec:mainresult}. We start by establishing an intertwining relationship, through the Poisson kernel with parameter $\lambda x>0$, between the generator of a CSBP (which may be immortal, namely its branching mechanism $\psi$ can be negative) and that of a specific $\mathbb{R}_+\times \mathbb{Z}_+$-valued branching process $(X^\lambda,L^{\lambda})$, see Theorem~\ref{thm1:main}. We show then that a relationship of the form \eqref{intro:intertwining} holds between $P_t$ the semigroup of $(X^\lambda,L^{\lambda})$ and $Q_t$ the semigroup of the CSBP$(\psi)$, see Theorem~\ref{thm:main}. The process $(L_t^{\lambda},t\geq 0)$ is the so-called \textit{skeleton}. The decomposition of the CSBP follows from an application of Pitman-Rogers theorem, see  Theorem~\ref{thmRG}, and states the following:  if \(X^\lambda\) starts from \(x > 0\) and the initial distribution of the skeleton \(L^\lambda\) has a Poisson law with parameter \(\lambda x\), the first coordinate projection \(X^\lambda\) is a  CSBP$(\psi)$ started from \(x\) and for any $t\geq 0$, the law of $L_t^{\lambda}$, conditionally on $X_t^{\lambda}$, is Poisson with parameter $\lambda X_t^{\lambda}$.

The dynamics of \((X^\lambda, L^\lambda)\) depends on whether $\lambda>\rho$, $\lambda=\rho$ or $\lambda<\rho$. When \(\lambda =\rho\), we recover the two-type process studied in \cite{BF08} for which the skeleton is the prolific discrete tree. When \(\lambda >\rho\), a death term arises along the skeleton and various results from \cite{DW07} are reobtained. The intertwining approach also enables us to investigate the setting $\lambda<\rho$, the decomposition here involves the Esscher transform \textit{at the left} of $\rho$, a scenario that, to our knowledge, has not been previously studied. As we shall see, there is a significant change in the dynamics: the skeleton is no longer autonomous, and the continuous mass  now generates discrete-type individuals through a proportional birth term. 
\\

Next, we observe that if the CSBP explodes without being killed, it does so simultaneously with any of its skeletons, see Proposition~\ref{prop:explosion}. Last, we establish that, for any branching mechanism (including the explosive and immortal ones), the skeletons, once rescaled by $\lambda$, converge weakly in the Skorokhod topology toward the CSBP, as $\lambda$ goes to $\infty$, see Theorem~\ref{thm:convergence}.
\\

\noindent \textbf{Notation.} We set $\mathbb{R}_+:=[0,\infty)$ and let $\mathbb{N}$ and $\mathbb{Z}_+$ be  respectively the sets of positive integers and non-negative integers. Denote by $C_0(\mathbb{R}_+)$ and $C_0(\mathbb{R}_+\times \mathbb{Z}_+)$ the spaces of continuous functions on $\mathbb{R}_+$ and $\mathbb{R}_+\times \mathbb{Z}_+$, respectively, that are vanishing at infinity. Denote by $C^2_0(\mathbb{R}_+)$, the functions that are twice-differentiable with first two derivatives in $C_0(\mathbb{R}_+)$.  For any function $f$ on $\mathbb{R}_+\times \mathbb{Z}_+$, such that $x\mapsto f(x,\ell) \in C_0^2(\mathbb{R}_+)$, the first two derivatives with respect to $x$ are denoted by $f'(x,\ell)$ and $f''(x,\ell)$.
\section{Preliminaries}\label{sec:preliminaries}
\subsection{$\mathbb{R}_+\times\mathbb{Z}_+$-valued continuous-time branching processes}
CSBPs have been introduced by Jirina \cite{zbMATH03270278}, Lamperti \cite{Lamperti2,Lamperti1} and Silverstein \cite{zbMATH03294035}. They are the scaling limits of Bienaymé-Galton-Watson processes
and represent the random evolution of a \textit{continuous} population. Two-dimensional branching processes with continuous-state space $\mathbb{R}_+\times \mathbb{R}_+$ have been defined  by Watanabe \cite{WA69}.

They are specific affine processes, see Duffie et al. \cite{DFS03} and Caballero et al. \cite{zbMATH06775439} and are also known to be strong solutions to certain stochastic differentials equations with jumps, see Barczy et al. \cite{zbMATH06433930}.  

Two-type branching processes taking values in \( \mathbb{R}_+ \times \mathbb{Z}_+ \) do not form a subclass of the branching processes with values in \(\mathbb{R}_+^2 \) studied in \cite{WA69}, for the same reason that branching processes taking values in \(\mathbb{Z}_+ \) do not form a subclass of the branching processes taking values in \(\mathbb{R}_+\) (the CSBPs): the former may have jumps of size $-1$, and not the latter.

The case where one component evolves in $\mathbb{Z}_+$ is thus structurally distinct and, to our knowledge, not reducible to any case studied in the literature.

While their general form will certainly not come as a surprise to the reader, they are central to this article and might prove useful in other contexts. We therefore begin by presenting them in detail. 
\\

Let $\gamma\in \mathbb{R}$ and  $b,\sigma,d, \kappa, \mathrm{k}\in \mathbb{R}_+$. Let $\pi(\ddr y,\ddr k)$ and $\rho(\ddr y,\ddr k)$ be two measures on $\mathbb{R}_+\times \mathbb{Z}_+$

such that
\begin{equation}\label{eq:integrabilitypirho}\int_{0}^{\infty}(1\wedge y^2)\, \pi(\ddr y,\{0\})+\sum_{k\geq 1}\int_{0}^{\infty}\pi(\ddr y,\ddr k)<\infty, \ \int_{0}^{\infty}\,(1\wedge y)\,\rho(\ddr y,\{0\})+ \sum_{k\geq 1}\int_{0}^{\infty}\,\rho(\ddr y,\ddr k)<\infty .\end{equation}
We call \textit{admissible} such parameters. Define the operator 
\begin{align}\label{eq:genL}
\mathscr{L}f(&x,\ell):=\gamma x f'(x,\ell)+b \ell f'(x,\ell)+\frac{\sigma^2}{2}xf''(x,\ell)-\kappa xf(x,\ell)-\mathrm{k} \ell f(x,\ell)\\
&+x\sum_{k\geq 0}\int_{\mathbb{R}_+}\left(f(x+y,\ell+k)-f(x,\ell)-y\ind_{(0,1)}(y)f'(x,\ell)\right)\pi(\ddr y,\ddr k)\nonumber\\
&+\ell\sum_{k\geq 0}\int_{\mathbb{R}_+}\left(f(x+y,\ell+k)-f(x,\ell)\right)\rho(\ddr y,\ddr k)+d\ell \big(f(x,\ell-1)-f(x,\ell)\big),\nonumber
\end{align}
with $(x,\ell) \in \mathbb{R}_+\times \mathbb{Z}_+$ and $f:(x,\ell)\mapsto f(x,\ell)$ a function  in the space $\mathcal{D}$:
\begin{align}\label{eq:domainofgeneratorL}
\mathcal{D}:=\Big\{ f:&(x,\ell)\in \mathbb{R}_+ \times \mathbb{Z}_+  \mapsto f(x,\ell)\in \mathbb{R}, \text{ such that: } \nonumber \\ 
&\forall \ell\in \mathbb{Z}_+,\ 
x\mapsto f(x,\ell) \in C^2_0(\mathbb{R}_+),
\quad \underset{x\rightarrow \infty}{\lim}x\big(|f(x,\ell)|+|f'(x,\ell)|+|f''(x,\ell)|\big)=0, \nonumber \\
&\qquad \qquad \qquad \qquad \qquad 
\text{ and } \forall x\in \mathbb{R}_+,  \underset{\ell \rightarrow \infty}{\lim} \ell \big(|f(x,\ell)|+|f'(x,\ell)|+|f''(x,\ell)|\big)=0 \Big\}.
\end{align} 
Consider the one-point compactification of $E:=\mathbb{R}_+\times \mathbb{Z}_+$, 

\vspace{2mm}
$E^{\Delta}:=\big(\mathbb{R}_+\times \mathbb{Z}_+\big) \cup \{\Delta\}$, with $\Delta:=\{(x,\ell): x+\ell=\infty\}$ and set $\mathscr{L}f(\Delta):=0$. 

\begin{theo}\label{thm:branchingbitype}  For any admissible parameters $(\gamma,b,\sigma,\kappa, \mathrm{k},\pi, \rho)$, there exists a unique $E^{\Delta}$-valued c\`adl\`ag strong Markov process $\mathbf{X}=(X,L)$, with cemetery state $\Delta$, solution to the martingale problem 
\begin{equation}\label{eq:MPbitype}\mathrm{MP}_{\mathbf{X}}(\mathscr{L},\mathcal{D}):\quad \forall f \in \mathcal{D}, \left(f(\mathbf{X}_t)-\int_{0}^{t}\mathscr{L}f(\mathbf{X}_s)\ddr s,t\geq 0\right) \text{ is a martingale}.
\end{equation}
Moreover, the semigroup of $\mathbf{X}$ satisfies for any $(x,n)$, $t\geq 0$ and $(q,r)\in (0,\infty)\times (0,1)$
\begin{equation}\label{dualsemigroup}
\mathbb{E}_{(x,n)}[e^{-qX_t}r^{L_t}]=e^{-x u_t(q,r)}f_t(q,r)^n,
\end{equation}
with $t\mapsto F_t(q,r):=\left(
u_t(q,r), f_t(q,r)\right)$, the unique solution to the two-dimensional o.d.e
\begin{align}\label{odeuf}
&\frac{\ddr }{\ddr t}
F_t(q,r)
=-\mathbf{\Psi}\big(F_t(q,r)\big),\\
&F_0(q,r)=\left(u_0(q,r),f_0(q,r)\right)=\left(q, r\right),
\end{align}
where $\mathbf{\Psi}=\big(\Psi_c, -\Psi_d\big)$, with for any $q\in (0,\infty)$ and $r\in [0,1)$, \begin{align}
\Psi_c(q,r)&:=\sum_{k\geq 0}\int_{\mathbb{R}_+}\left(e^{-qy}r^{k}-1+qy\ind_{(0,1)}(y)\right)\pi(\ddr y, \ddr k)-\gamma q+\frac{\sigma^2}{2}q^2-\kappa,\label{eq:Psic}\\
\Psi_d(q,r)&:=\sum_{k\geq 0}\int_{\mathbb{R}_+}\left(e^{-qy}r^{k+1}-r\right)\rho(\ddr y, \ddr k)-bqr+d(1-r)-\mathrm{k}r.\label{eq:Psid}
\end{align}
\end{theo}
\begin{rem}
The killing terms in \eqref{eq:genL} with parameters $\kappa$ and $\mathrm{k}$  can be interpreted as single jumps to the boundary $\infty$. Indeed for any $f\in \mathcal{D}$, 
\begin{center}
$-\kappa xf(x,\ell)=\kappa x\big(f(\infty,\ell)-f(x,\ell)\big) \text{ and } -\mathrm{k} \ell f(x,\ell)=\mathrm{k} \ell\big(f(x,\infty)-f(x,\ell)\big)$.
\end{center}
\end{rem}
\begin{proposition}\label{Fellerprop}
The semigroup $(P_t)$ of the process $\mathbf{X}$ satisfying \eqref{dualsemigroup} is Feller with absorbing state $\Delta$. For all $(q,r)\in (0,\infty)\times (0,1)$, set $f_{q,r}:(x,\ell)\mapsto e^{-qx}r^{\ell}$. The space \begin{equation}\label{eq:coreD} D:=\emph{Vect}\left\{f_{q,r}, q\in (0,\infty), r\in (0,1)\right\}, \end{equation} satisfies $P_tD\subset D$ and is a core for the generator $\mathscr{L}$ of $\mathbf{X}$. Moreover, $\mathbf{X}$ possesses the branching property. Specifically, its transition kernel $P_t\big((x,\ell),\cdot\big)$ satisfies:
\begin{equation}\label{branchingprop} \forall x,y\in [0,\infty),\ \forall n,m\in \mathbb{Z}_+, \qquad P_t\big((x+y,n+m),\cdot\big)=P_t\big((x,n),\cdot\big)\star P_t\big((y,m),\cdot\big),
\end{equation}
where $\star$ stands for the convolution of measures. \end{proposition}
Theorem \ref{thm:branchingbitype} and Proposition \ref{Fellerprop} are established in Section \ref{sec:constructionbitype}. We call $\mathbf{\Psi}$ the \textit{joint branching mechanism} of $\mathbf{X}$.
\begin{rem}
Any $\mathbb{R}_+\times \mathbb{Z}_+$-valued Markov process  satisfying the branching property \eqref{branchingprop} and the following condition on its semigroup 
\begin{center} $t\mapsto \mathbb{E}_{(x,n)}[e^{-qX_t}r^{L_t}]$ is differentiable at $0$, \end{center} will fall into the class studied here. We do not enter into this study here and refer for instance to Gihman and Skorokhod's book \cite[Chapter V]{GS04}. 
\end{rem}

Theorem \ref{thm:branchingbitype} covers the two classical settings, namely the discrete-state branching processes and those in continuous-state space, as we now show. We recall some well-known facts about them. More background can be found in the books by Harris \cite{zbMATH02001586}, Kyprianou \cite{Kyprianoubook} and Li \cite{zbMATH07687769}. 
\vspace*{2mm}

\subsubsection{One-dimensional discrete-state branching processes.}
Let $\Psi_d$ as in \eqref{eq:Psid}, with $b\equiv 0$ and a measure $\rho$ shrinking to $\rho(\ddr y,\ddr k)=\delta_0(\ddr y)\rho_d(\ddr k)$, for some finite measure $\rho_d(\ddr k)$ on $\mathbb{N}$. In this case there is no dependence on the variable $q$ in $\Psi_d$ and the function $$\varphi: [0,1)\ni r\to \Psi_d(0,r)=\Psi_d(q,r)=\sum_{k\geq 1}(r^{k+1}-r)\mu(k)+d(1-r)-\mathrm{k}r,$$ is the mechanism of a continuous-time Markov branching process with reproduction measure $\mu:=d\delta_{-1}+\rho_d+\mathrm{k}\delta_{\infty}$. 
Furthermore in the case $\Psi_c\equiv 0$, the component $X$ in the process $(X,L)$ is degenerated to the constant process and $L$ is a classical discrete branching process with reproduction measure $\mu$. Notice that when $\mathrm{k}=0$, $$\varphi'(1-)=\sum_{k\in \mathbb{Z}_+\cup\{-1\}}k\mu(k)=\sum_{k\in \mathbb{N}}k\mu(k)-d\in (-\infty,\infty],$$
and $L$ is supercritical if and only if $\varphi'(1-)>0$. We say that $L$ is immortal if there is no death in its dynamics, i.e. $d\equiv 0$.
A necessary and sufficient condition for the process $L$ to be non-explosive is $\int^{1}\frac{\ddr x}{|\varphi(x)|}=\infty$, see \cite{Harris}.

\subsubsection{One-dimensional continuous-state branching processes}\label{sec:csbps} 
Let $\Psi_c$ as in \eqref{eq:Psic} with $\pi(\ddr y,\ddr k)=\delta_0(\ddr k)\nu(\ddr y)$, for a certain L\'evy measure $\nu$ on $(0,\infty)$, such that $\int_0^{\infty}(1\wedge y^2)\nu(\ddr y)<\infty$. In this case there is no dependence on the variable $r$ in $\Psi_c$, and the function \begin{equation}\label{eq:branchingmechanismpsi}
\psi:(0,\infty) \ni q\mapsto \Psi_c(q,1)=\Psi_c(q,r)=\frac{\sigma^2}{2} q^{2}-\gamma q+ \int_{0}^{\infty}\big(e^{-qy}-1+qy\ind_{(0,1)}(y)\big)\nu(\ddr y)-\kappa,
\end{equation}
is the branching mechanism of a CSBP with parameters $(\frac{\sigma^2}{2},\gamma,\nu, \kappa)$. In this case the process $L$ is \textit{autonomous} as its dynamics does not depend on the component $X$. 
\vspace*{1.5mm}
\\
Furthermore in the case $\Psi_d\equiv 0$, the component $L$ degenerates to a constant process and $X$ is a classical CSBP$(\psi)$.  When $\kappa=0$, we recall that $X$ is said to be supercritical if $\psi'(0+)\in [-\infty,0)$, critical if $\psi'(0+)=0$ and subcritical when $\psi'(0+)>0$. 
The CSBP $X$ is immortal, i.e. $X_t\underset{t\rightarrow \infty}{\longrightarrow} \infty$ a.s. if and only if $\psi\leq 0$, see e.g. \cite[Chapter 12]{Kyprianoubook}.

\vspace*{1.5mm}
Theorem \ref{thm:branchingbitype} and Proposition \ref{Fellerprop} applied in this special setting ensure the following facts:  the CSBP$(\psi)$, $X$, is a $[0,\infty]$-valued càdlàg Feller process, with $\infty$ as absorbing state and its generator is 
 
\begin{equation}\label{eq:G}
\mathcal{G}f(x):=\frac{\sigma^2}{2}xf''(x)+\gamma x f'(x)+x\int_{0}^{\infty}\big(f(x+y)-f(x)-yf'(x)\ind_{(0,1)}(y)\big)\nu(\ddr y)-\kappa xf(x),
\end{equation}
acting on
\begin{equation}\label{Dc}
\mathcal{D}_c:=\left\{f\in C^{2}_0: \underset{x\rightarrow \infty}{\lim} x\big(|f(x)|+|f'(x)|+|f''(x)|\big)=0\right\}.
\end{equation} 
Denote by $Q_t$ the semigroup associated with $X$. Setting $e_q(x):=e^{-qx}$ for all $q>0$ and $x\in [0,\infty)$, one has $\mathcal{G}e_q(x)=x\psi(q)e_q(x) \text{ and }$
\[Q_te_q(x)=e^{-xu_t(q)} \text{ with } \frac{\ddr }{\ddr t}u_t(q)=-\psi(u_t(q)), u_0(q)=q.\] 
Moreover, the space
\begin{equation}\label{Dexpoc}
D_c:=\mathrm{Vect}\{e_q(\cdot),q\in (0,\infty)\}
\end{equation}
is a core and the semigroup $(Q_t)$  uniquely satisfies the backward Kolmogorov equation
\[\forall f\in D_c,\quad \frac{\ddr }{\ddr t}Q_tf=\mathcal{G}Q_tf.\]
 When $\kappa=0$, we recall that $X$ is said to be supercritical if $\psi'(0+)<0$, critical if $\psi'(0+)=0$ and subcritical when $\psi'(0+)>0$. The CSBP $X$ is immortal, i.e. $X_t\underset{t\rightarrow \infty}{\longrightarrow} \infty$ a.s. if and only if $\psi\leq 0$, see e.g. \cite[Chapter 12]{Kyprianoubook}. A necessary and sufficient condition for the continuous-state process $X$ to be non-explosive is $\int_{0}\frac{\ddr x}{|\psi(x)|}=\infty$, see Grey \cite{zbMATH03471400}. 
\\

We now give natural conditions under which the first and second  coordinates are autonomous (discrete or continuous) branching processes.

\begin{proposition}[Explosion and autonomous coordinates]
\label{prop:autonomous}\

\begin{enumerate}
\item
In case $\pi(\ddr y,\ddr k)= \nu(\ddr y) \delta_0(\ddr k)$ for a measure  $\nu(\ddr y)$ on $(0,\infty)$ such that $\int_0^\infty (1\wedge y^2) \nu(\ddr y)<\infty$, and under the condition $\int_{0}\frac{\ddr q}{|\Psi_c(q,1)|}=\infty$, 
the coordinate $(L_t)$  is a discrete branching process with branching mechanism
$$\Psi_d(0,r)=\sum_{k \geq 1} (r^{k+1}-r) \rho(\R_+,\{k\}) +d(1-r)-\mathrm{k}r.$$ 
 
\item 

In case $\rho(\ddr y,\ddr k)= \delta_0(\ddr y) \rho_d(\ddr k)$ for a finite measure  $\rho_d(\ddr k)$ on $\Z_+$, $b=0$, and under the condition
$\int^1\frac{\ddr r}{|\Psi_d(0,r)|}=\infty$, the coordinate $(X_t)$ is a CBSP with branching mechanism 
\begin{equation*}
\Psi_c(q,1)=\frac{\sigma^2}{2} q^{2}-\gamma q+ \int_{0}^{\infty}\big(e^{-qy}-1+qy\ind_{(0,1)}(y)\big)\pi(\ddr y,\Z_+)-\kappa.
\end{equation*}
\end{enumerate}
\end{proposition}
The integral conditions in Proposition \ref{prop:autonomous} ensure that the autonomous coordinate does not explode, hence preventing the situation where the other coordinate stops evolving by being in the cemetery point. The proposition is proved at the end of Section \ref{sec:constructionbitype}.

\section{Main results}

\label{sec:mainresult}
Let $\psi$ be a branching mechanism with quadruplet by $(\sigma,\gamma,\nu,\kappa)$, see \eqref{eq:branchingmechanismpsi}. The largest root of $\psi$ is denoted by \[\rho=\psi^{-1}(0):=\sup\{x>0:\psi(x)<0\}\in [0,\infty].\] 

As explained in the Introduction, we shall see hereafter that the CSBP$(\psi)$ hides a family of  $\mathbb{R}_+\times \mathbb{Z}_+$-valued branching processes. 

For any $\lambda \in (0,\infty)$, the Esscher transform of $\psi$ at $\lambda$, is given by $\psi_{\lambda}(\cdot):=\psi(\lambda+\cdot)-\psi(\lambda)$, see Figure \ref{figure:Esscher}. This defines a new branching mechanism with no killing term, i.e. $\psi_{\lambda}(0+)=0$. Specifically $\psi_\lambda$ takes the explicit form
\begin{equation}\label{eq:esscher}
\psi_\lambda(q)=\frac{\sigma^2}{2}q^2+\psi'(\lambda)q+\int_{0}^\infty ye^{-\lambda y}\nu(\ddr y)\left(e^{-qy}-1+qy\right).
\end{equation}
We denote by $\mathcal{G}^{\psi_{\lambda}}$ the infinitesimal generator of the CSBP$(\psi_\lambda)$.
\\
\begin{figure}[h!]
\centering \noindent
\includegraphics[height=.20	 \textheight]{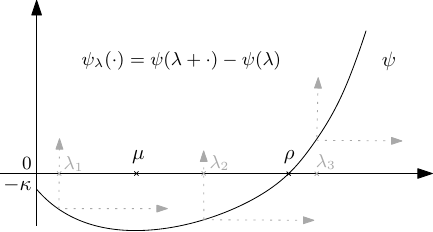}
\caption{A supercritical mechanism and its Esscher transforms}
\label{figure:Esscher}
\end{figure}

Define the Poisson kernel $K$: 
\begin{equation}\label{Poissonkernel}K(x,\ell):=e^{-\lambda x}\frac{(\lambda x)^{\ell}}{\ell !}, \forall x\in [0,\infty), \forall \ell \in \mathbb{Z}_+.
\end{equation}
For any function $f:(x,\ell)\mapsto f(x,\ell)$, belonging to $\mathcal{D}$, we set 
\[Kf(x):=\sum_{\ell=0}^{\infty}e^{-\lambda x}\frac{(\lambda x)^{\ell}}{\ell !}f(x,\ell),\text{ and }\mathcal{G}^{\psi_\lambda}f(x,l):=\mathcal{G}^{\psi_\lambda}f_\ell(x) \text{ with } f_{\ell}:x\mapsto f(x,l).\]

We will introduce an operator $\mathcal{H}$, acting on $\mathcal{D}$, that will satisfy an \textit{algebraic} intertwining relationship with the one-dimensional generator $\mathcal{G}$ via the kernel $K$.
\vspace*{2mm} 

To begin, we define an operator $\mathcal{J}$ that plays a central role. For any $f\in \mathcal{D}$ and $(x,\ell)\in \mathbb{R}_+\times \mathbb{Z}_+$, set
\begin{align}\label{J=massonnodes}\mathcal{J}f(x,\ell):= \ell\sum_{k \geq 0} \int_{(0,\infty)} &\big(f(x+y,\ell+k)-f(x,\ell)\big)\mathrm{s}(\ddr y,\ddr k) \nonumber \\
&+\frac{\sigma^2}{2}\lambda \ell[f(x,\ell+1)-f(x,\ell)]+\sigma^2\ell f'(x,\ell),
\end{align}
with
\begin{equation}\label{sdydk}
\mathrm{s}(\ddr y,\ddr k):=ye^{-\lambda y}\frac{(\lambda y)^{k}}{(k+1)!}\nu(\ddr y)\ind_{\mathbb{Z}_+}(k).\end{equation}

\begin{theo}[Algebraic intertwining]\label{thm1:main} 
Let $\mathcal{H}$ be the operator defined on $\mathcal{D}$ as follows:
\begin{itemize}
\item[i)] When $\lambda \geq \rho$, set
\begin{align}\label{Hi}
\mathcal{H}f(x,\ell):=-\kappa xf(x,\ell)+\mathcal{G}^{\psi_{\lambda}}f(x,\ell)+\mathcal{J}f(x,\ell)+ \ell \;  \frac{\psi(\lambda)}{\lambda} [f(x,\ell-1)-f(x,\ell)].
\end{align}
\item[ii)] When $\lambda \leq \rho$, set
\begin{align}\label{Hii}
\mathcal{H}f(x,\ell):=-\kappa xf(x,\ell)+\mathcal{G}^{\psi_{\lambda}}f(x,\ell)&+ \mathcal{J}f(x,\ell)- x\psi(\lambda)[f(x,\ell+1)-f(x,\ell)].
\end{align}
\end{itemize}
In both cases, one has \begin{equation}\label{eq:intertwiningHKG} K\mathcal{H}f(x)=\mathcal{G}Kf(x),\ \forall x\geq 0.
\end{equation}
\end{theo}
The proof of Theorem \ref{thm1:main} is  provided in Section \ref{sec:proofmainthm1}. Notice that in case i), $\psi(\lambda)\geq 0$ while in case ii) $\psi(\lambda)\leq 0$. This ensures that the multiplicative factor preceding the last term in \eqref{Hi} and \eqref{Hii} is non-negative.

\vspace*{2mm}

The operator $\mathcal{H}$ given by \eqref{Hi}-\eqref{Hii} is the generator of a  $\mathbb{R}_+\times \mathbb{Z}_+$-valued branching process $\mathbf{X}:=(X^{\lambda},L^{\lambda})$. We  represent and explain further its dynamics in Figure \ref{figure:dessin}.  The jump measures $\rho$ and $\pi$, as they appear in the general form of the generator \eqref{eq:genL}, are given as follows
\vspace*{-2mm}
\begin{equation}\label{eq:rho}
\rho(\ddr y,\ddr k):=\mathrm{s}(\ddr y, \ddr k)+\frac{\sigma^2}{2}\lambda\delta_{0}(\ddr y)\delta_1(\ddr k)
\end{equation} and 
\vspace*{-2mm}
\begin{equation}\label{eq:pi}\pi(\ddr y, \ddr k):=\begin{cases}ye^{-\lambda y}\nu(\ddr y)\delta_0(\ddr k) & \text{ if } \lambda \geq \rho,\\
ye^{-\lambda y}\nu(\ddr y)\delta_0(\ddr k)-\psi(\lambda)\delta_{0}(\ddr y)\delta_1(\ddr k) & \text{ if } \lambda <\rho.\end{cases}
\end{equation}
The following corollary gives an analytic expression for the joint branching mechanism $\mathbf{\Psi}=(\Psi_c,-\Psi_d)$ of $(X^{\lambda},L^{\lambda})$, thereby elucidating its connection to $\psi$. 

\begin{cor}[Joint branching mechanism]\label{cor:identifyjointbranchingmech} For any $\lambda>0$, 

\begin{equation}\label{eq:relationship-psiPsi}
\psi\big(q+\lambda(1-r)\big)=\Psi_c(q,r)+\lambda\Psi_d(q,r), \ \forall (q,r)\in (0,\infty)\times (0,1).
\end{equation}
Moreover,
 
\begin{align}
i) \text{ When } \lambda\geq \rho: \quad \Psi_c(q,r)&=\psi_\lambda(q)-\kappa, \label{Psic-(i)}\\
\Psi_d(q,r)&=\sum_{k\in \mathbb{Z}_+}\int_{(0,\infty)}\!\!\left(e^{-qy}r^{k+1}-r\right)\rho(\ddr y,\ddr k)+\frac{\psi(\lambda)}{\lambda}(1-r)-\sigma^2qr \label{Psid-(i)}\\
&=\frac{\psi\left(q+\lambda(1-r)\right)-\psi(q+\lambda)+\psi(\lambda)+\kappa}{\lambda}. \label{Psid-(i)2}\\
ii) \text{ When } \lambda\leq \rho: \quad \Psi_c(q,r)&=\psi_{\lambda}(q)-\kappa-\psi(\lambda)(r-1), \label{Psic-(ii)}\\
\Psi_d(q,r)&=\sum_{k\in \mathbb{Z}_+}\int_{(0,\infty)}\!\!\left(e^{-qy}r^{k+1}-r\right)\rho(\ddr y,\ddr k)-\sigma^2qr \label{Psid-(ii)}\\
&=\frac{\psi\left(q+\lambda(1-r)\right)-\psi(q+\lambda)+r\psi(\lambda)+\kappa}{\lambda}.\label{Psid-(ii)2}
\end{align}
\end{cor}
The proof of Corollary \ref{cor:identifyjointbranchingmech} is in Section \ref{sec:proofprop1}.
\begin{rem}
Notice that when $\rho\in (0,\infty)$, the map $\lambda\mapsto (\Psi_c,\Psi_d)$ is continuous (for the uniform norm) at $\lambda=\rho$, since $\psi$ is continuous and $\psi(\rho)=0$.
\end{rem} 
For any $\lambda>0$ and $x\in [0,\infty)$, denote by $\mathrm{Poi}(\lambda x)$ the Poisson law with parameter $\lambda x$ and notice that $K(x,\ell)=\mathbb{P}(\mathrm{Poi}(\lambda x)=\ell)$ for all $\ell \in \mathbb{Z}_+$.  Let $(P_t)$ be the semigroup of the branching process $(X^{\lambda},L^{\lambda})$ with mechanism $(\Psi_c,-\Psi_d)$, and $(Q_t)$ be the semigroup of a CSBP$(\psi)$. Call $(\mathcal{F}^{X^{\lambda}}_t)$ the usual augmentation of the natural filtration of $X^{\lambda}$, see e.g. Revuz-Yor's book \cite[pages 45 and 93]{zbMATH02150787}. 

\begin{theo}[Intertwined semigroups and skeleton decomposition]\label{thm:main}
For any function $f\in C_0(\mathbb{R}_+\times \mathbb{Z}_+)$, one has

\begin{equation}\label{eq:intertwinedsemigroups}
\Lambda P_tf(x)=Q_t\Lambda f(x),\quad \forall t,x\geq 0,\end{equation}
where for all $x\in \bar{\mathbb{R}}_+$ and $\ell\in \bar{\mathbb{Z}}_+$,
\begin{equation}
\label{eq:defLambda}
\Lambda\big(x,(x,\ell)\big):=K(x,\ell)
\ind_{\{x<\infty\}}+\delta_{\Delta}\big(x,\ell)\ind_{\{x=\infty\}}.
\end{equation}

Moreover, if $(X^{\lambda},L^{\lambda})$ has for initial law $X^{\lambda}_0=x$ and $L^{\lambda}_0\overset{\text{law}}{=}\mathrm{Poi}(\lambda x)$, then the following holds:
\begin{enumerate}
\item The process $X^{\lambda}$ is a CSBP$(\psi)$ started from $x$.

\item For all $\ell\in \mathbb{Z}_+$ and any $(\mathcal{F}^{X^{\lambda}}_t)$-stopping time $\tau$,  \[\mathbb{P}(L^{\lambda}_{\tau}=\ell|\mathcal{F}^{X^{\lambda}}_{\tau})=K(X^{\lambda}_{\tau},\ell) \text{ a.s.} \text{ on the set } \{X_\tau< \infty\}. \]
\end{enumerate}
\end{theo}
Theorem \ref{thm:main} is established in  Section \ref{sec:proofmainthm} and relies on Pitman-Rogers theorem, 
see Theorem \ref{thmRG}.
\vspace*{1mm}
\\ 
Let us provide a more detailed explanation of the statement of Theorem \ref{thm:main}. Consider first the case of a supercritical branching mechanism $\psi$, i.e. $\psi'(0+)<0$, with no killing ($\kappa=0$) and with $\rho\in (0,\infty)$. Regarding the event of asymptotic extension, one has:
\vspace*{-2mm}
\begin{center}
$\mathbb{P}_x(X_t\underset{t\rightarrow \infty}{\longrightarrow}0)=e^{-x\rho}=\P(L^{\rho}_0=0)\in (0,1)$,
\end{center} 
\vspace*{-3mm}
see e.g. \cite[Theorem 12.7]{Kyprianoubook} for the first equality. In the case $\lambda=\rho$, the process $(X^{\lambda},L^{\lambda})$ coincides with the process studied in \cite[pages 722-723]{BF08}. It is a branching process with two types: the \textit{prolific} individuals, which are represented along the discrete component $(L^\rho_t,t\geq 0)$ and the non-prolific ones, evolving in $\mathbb{R}_+$. Since $\psi(\rho)=0$, the discrete branching process $(L^\rho_t,t\geq 0)$ has no death term, see \eqref{Psid-(i)}, and is thus immortal. The event of having no prolific individuals in the CSBP, $\{L_0^\rho=0\}$, also matches with the event of asymptotic extinction. Furthermore, the continuous dynamics is governed by $\psi_\rho$, which is the mechanism of a CSBP$(\psi)$ conditioned on its extinction, see \cite[Exercice 12.4]{Kyprianoubook}. Finally, Theorem~\ref{thm:main} asserts that the projection onto the continuous component results in a CSBP$(\psi)$.

More generally, for any branching mechanism $\psi$ such that $\rho\in [0,\infty)$, when $\lambda\geq \rho$, case i), the branching mechanism $\psi_\lambda$ is \textit{subcritical}, see Figure \ref{figure:Esscher}. The CSBPs$(\psi_\lambda)$, with generator $\mathcal{G}^{\psi_\lambda}$, are thus converging towards $0$ almost surely. Notice that $\Psi_c$, see \eqref{Psic-(i)} does not depend on the variable $r$. Thus, the discrete component $L^{\lambda}$ is autonomous. In particular, conditionally on $\{L^{\lambda}_0=0\}$, $L_t^{\lambda}=0$ for all $t\geq 0$ a.s. and  $X^\lambda$ is a CSBP($\psi_\lambda$) started from $x$.

We now explain the dynamics of the process $\mathbf{X}=(X^{\lambda},L^{\lambda})$ for both scenarios: $\lambda>\rho$ and  $\lambda<\rho$, see Figure \ref{figure:dessin}.   We call \textit{skeleton}, the discrete component $L^{\lambda}$, and refer to the process $\mathbf{X}$ as a skeleton decomposition.

\begin{figure}[h!]
\centering \noindent
\includegraphics[height=.30	 \textheight]{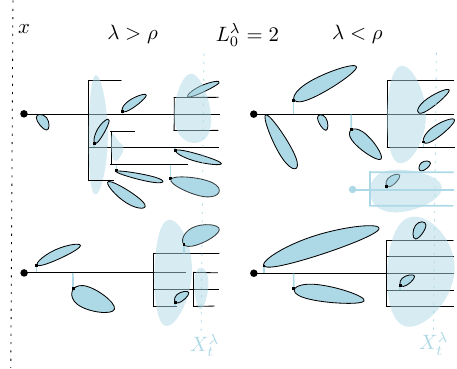}
\caption{Schematic representation of the intertwined two-type branching process}
\label{figure:dessin}
\end{figure}

The skeleton is depicted by lines, while the blue parts represent the continuous mass. The bubbles encircled in black represent CSBPs($\psi_\lambda$) that are grafted along each line. The CSBP starting from a black square starts from a macroscopic mass governed by the Lévy measure\footnote{One has $\int_0^1y\,\mathrm{s}(\ddr y,\{0\})<\infty$} $\mathrm{s}(\ddr y,\{0\})=ye^{-\lambda y}\nu(\ddr y)$. The bubbles that stick to the skeleton only exist if $\sigma>0$. They arrive on each line at rate $\sigma^2$ and represent CSBPs started from an infinitesimal mass \footnote{A rigorous formalisation requires to work with the canonical measure, see e.g. \cite{zbMATH05598047}. We shall not need it in our construction here}. 

At any splitting time of the skeleton, say into $k\geq 1$ new lines, a continuous mass, governed by $\mathrm{s}(\ddr y,\{k\})$, is spread and starts evolving as a CSBP($\psi_\lambda$). This is depicted by the blue shadow behind the lines. In the case $\psi(0+)=-\kappa<0$, the continuous mass is furthermore killed at rate $\kappa x$. In other words, we graft CSBPs with mechanism $\psi_\lambda-\kappa$. In any case, the skeleton and the grafted CSBPs accomodate in a way so that the first coordinate process $X^{\lambda}$ keeps the law of a CSBP$(\psi)$.

\begin{itemize}
\item[i)] When $\lambda>\rho$, a death term, with rate $\frac{\psi(\lambda)}{\lambda}$, in the skeleton $L^{\lambda}$ emerges, see \eqref{Psid-(i)}. The latter has therefore leaves as represented in Figure \ref{figure:dessin}.  Heuristically, since $\psi_\lambda'(0+)=\psi'(\lambda)$ and $\psi(\lambda)/\lambda$ increase on $[\rho,\infty)$, the greater $\lambda$ is, the more subcritical are the CSBP$(\psi_\lambda)$ and the more leaves has the skeleton.

\vspace*{1mm}

\item[ii)] When $\lambda<\rho$, case ii), the dynamics changes as the continuous mass $X^{\lambda}$ also begets discrete individuals. It is acting on the skeleton through an additional linear birth term with positive rate $-\psi(\lambda)$. In this setting, the skeleton is immortal (it has no leaf) and is not autonomous, as $\Psi_c$ depends on the variable $r$, see \eqref{Psic-(ii)}. 

Let $\mu:=\text{argmin}{\psi}\in [0,\infty]$ be the location of the minimum of $\psi$, see Figure \ref{figure:Esscher}. The branching mechanism $\psi_\lambda$ is subcritical, critical or supercritical when respectively $\lambda>\mu$, $\lambda=\mu$ and $\lambda<\mu$. 
When $\lambda<\mu$, not only does the continuous mass generate new lines, but the grafted CSBPs must also be supercritical to ensure that the total continuous mass process $X^{\lambda}$ retains the law of the CSBP$(\psi)$.

\end{itemize}  
Notice that when the CSBP$(\psi)$ is \textit{immortal}, i.e. it tends to $\infty$ a.s., one has $\rho=\mu=\infty$ and only the decomposition (ii), with supercritical grafted CSBPs, makes sense. A classical example in this framework is the stable mechanism $\psi(\lambda)=-c\lambda^{\alpha}$ with $\alpha\in (0,1), c>0$.
\begin{rem}[Conditional law of $X^\lambda$ given  $L^{\lambda}$] Let $n\geq 0$ and let $A$ be a Borel set. Since conditionally on $X_t^{\lambda}$,  $L_t^{\lambda}\overset{\text{law}}{=} \mathrm{Poi}(\lambda X_t^{\lambda})$, we see  that:
\[\mathbb{P}(X^{\lambda}_t\in A|L_t^{\lambda}=n)=\frac{\mathbb{P}\big(X_t^{\lambda}\in A,L^{\lambda}_t=n\big)}{\mathbb{P}(L_t^{\lambda}=n)}=\frac{\mathbb{E}\big[e^{-\lambda X_t^{\lambda}}(X^{\lambda}_t)^n \ind_A(X^{\lambda}_t)\big]}{\mathbb{E}\big[e^{-\lambda X_t^{\lambda}}(X^{\lambda}_t)^n\big]}.\]

\end{rem}

\begin{rem}[A one-dimensional intertwining] 
The identities \eqref{eq:intertwiningHKG} and \eqref{eq:intertwinedsemigroups} encapsulate an intertwining relationship between the CSBP and its skeleton alone. Let $G^{\lambda}$ and $Q^{\lambda}_t$ denote the infinitesimal generator and the semigroup of the process $L^{\lambda}$. For any bounded function $f:\mathbb{Z}_+\mapsto \mathbb{R}_+$, set $Kf(x)=\sum_{\ell \geq 0}f(\ell)\frac{(\lambda x)^{\ell}}{\ell!}e^{-\lambda x}$, one has for all $x\geq 0$ and $t\geq 0$,
\[KG^{\lambda}f(x)=\mathcal{G} Kf (x) \text{ and } KQ_t^{\lambda}f(x)=Q_tK f(x).\]
We emphasize that these identities alone do not establish any direct coupling between the processes governed by the semigroups $Q_t^{\lambda}$ and $Q_t$. In general, finding such a coupling is challenging - see the pioneering article \cite{zbMATH04192817} and for instance \cite{zbMATH07206399}. Within the framework of Theorem \ref{thm:main}, the intertwining relationship is directly related to a function (the projection on the first coordinate) and Pitman-Rogers theorem, see Theorem \ref{thmRG}, applies.
\end{rem}

As an application of Theorem \ref{thm1:main}, we now study the phenomenon of explosion of a CSBP. 

Notice that for all $\lambda\in (0,\infty)$, $\psi_\lambda'(0+)=\psi'(\lambda)\in (-\infty,\infty)$, the grafted CSBPs have thus finite mean. In particular, none is exploding. In the case $\kappa=0$, if the CSBP$(\psi)$ can explode, i.e. $\int_{0}\frac{\ddr x}{|\psi(x)|}<\infty$, see e.g. \cite[Theorem 12.3]{Kyprianoubook},  its explosion occurs simultaneously with any of its skeleton $L^{\lambda}$. Heuristically, the only way to reach infinity by accumulation of large jumps is to graft CSBPs on an infinite number of lines. This is stated in the next proposition whose proof will appeal Theorem \ref{thm:main}.

\begin{prop}[Simultaneous explosion]
\label{prop:explosion} If $(X^{\lambda},L^{\lambda})$ has for initial law $X^{\lambda}_0=x$ and $L^{\lambda}_0\overset{\text{law}}{=}\mathrm{Poi}(\lambda x)$, then the following identity of events holds a.s. 
\[\{L_t^{\lambda}=\infty\}=\{X_t^{\lambda}=\infty\}, \text{ for all } t\geq 0.\]
In particular, when $\int_{0}\frac{\ddr x}{|\Psi(x)|}<\infty$, explosion has positive probability and occurs simultaneously in the CSBP($\psi$), $X^{\lambda}$, and  the skeleton $L^{\lambda}$. 
\end{prop}
\begin{rem}
Notice that the event of simultaneous extinction of both types is not almost sure, indeed for any $\lambda>0$, when $L_0^\lambda\overset{\text{law}}{=}\mathrm{Poi}(\lambda x)$ and $X^{\lambda}_0=x$, we have by Theorem \ref{thm:main}-(2): 
\begin{center}
$\mathbb{P}(L^{\lambda}_t=0)=\mathbb{E}[e^{-\lambda X^{\lambda}_t}]>\mathbb{P}(X^{\lambda}_t=0)$, $\forall t\geq 0$.
\end{center}
\end{rem}
The proof of Proposition \ref{prop:explosion} is in Section \ref{sec:proofprop2}.
\vspace*{2mm}

We study now the family of skeletons, with $\lambda\geq \rho$, and will establish that suitably renormalized, they converge in the sense of Skorokhod, towards a CSBP$(\psi)$, see Theorem \ref{thm:convergence} below. 

We start by providing explicitely the reproduction law of the autonomous skeletons. The following formula can be found in \cite{DW07,FFK19,zbMATH07940610} in the case $\kappa=0$. 
We will shed some light on the calculations behind. 

\begin{prop}[Offspring distribution of skeletons]\label{prop:mechanismskeleton} 
Assume $\lambda\geq \rho$. In this case, the skeleton $(L_t^{\lambda},t\geq 0)$ is autonomous and its branching mechanism is given by
\begin{equation}\label{varphi}\varphi_\lambda(r):=\Psi_d(0,r)=\frac{1}{\lambda}\Big(\psi\big(\lambda(1-r)\big)+\kappa\Big)=\psi'(\lambda)\sum_{k=-1}^{\infty}(r^{k+1}-r)p_k^{\lambda},
\end{equation}
with $(p_{k}^{\lambda},k\geq -1)$ the probability measure given by
\begin{equation}\label{pk} p_{-1}^{\lambda}=\frac{1}{\lambda\psi'(\lambda)}\psi(\lambda),\quad  p_0^{\lambda}=0,\quad p^{\lambda}_k=\frac{1}{\lambda\psi'(\lambda)}\left(\frac{\sigma^2\lambda^2}{2}\ind_{\{k=1\}}+\int_{0}^{\infty}\!\!e^{-\lambda y}\frac{(\lambda y)^{k+1}}{(k+1)!}\nu(\ddr y)\right), \forall
 k\geq 1.\end{equation}
\end{prop}
The proof of Proposition \ref{prop:mechanismskeleton} is in Section \ref{sec:proofprop:mechanismsskeleton}.
\begin{rem} Notice that $\varphi'_\lambda(1-)=-\psi'(0+)$. In particular, $L^{\lambda}$ is (sub)critical if and only if the CSBP($\psi$) is (sub)critical. Observe also that when $\lambda<\rho$, the formula \eqref{pk} cannot define an offspring distribution as at least one of the following holds $\psi'(\lambda)<0$ or $\psi(\lambda)<0$, which violates the condition $p_k\geq 0$ for all $k$. The formula \eqref{pk} is reminiscent to the coalescence rates of ancestral lineages in CSBPs backwards in time, see Foucart et al. \cite[Theorem 5.10]{zbMATH07142897} and Johnston and Lambert \cite{zbMATH07789648}.
\end{rem}
\begin{rem}[Binomial intertwining]
The family of continuous-time Galton-Watson processes $(L^\lambda,\lambda\geq\rho)$ satisfies an intertwining relationship with one another through a binomial kernel. This can be seen as a counterpart of the Bernoulli leaf coloring investigated in \cite[Section 4.1]{DW07}.  Recall $G^{\lambda}$ and $Q^{\lambda}_t$, the generator and semigroup of $L^{\lambda}$. For all $\mu>\lambda>\rho$ and any bounded function $f:\mathbb{Z}_+\rightarrow \mathbb{R}$, we have for all $n\geq 0$ and $t\geq 0$,
\[G^{\mu}B_{\lambda/\mu}f(n)=B_{\lambda/\mu}G^{\lambda}f(n) \text{ and } Q^{\mu}_tB_{\lambda/\mu}f(n)=B_{\lambda/\mu}Q_t^{\lambda}f(n)\]
with $B_{\lambda/\mu}f(n):=\sum_{k=0}^{n}\binom{n}{k}(\lambda/\mu)^k(1-\lambda/\mu)^{n-k}f(k)$.

\end{rem}
\vspace*{2mm}
The next theorem shows that for any branching mechanism $\psi$, the skeletons converge towards the CSBP$(\psi)$ in the Skorokhod sense.
\begin{theo}\label{thm:convergence} Let $x\in [0,\infty)$. Denote by $(X_t(x),t\geq 0)$ a CSBP$(\psi)$ starting from $x$, and  let $(L^{\lambda}_t(x),t\geq 0)$ be a discrete branching Markov process with mechanism $\varphi_\lambda$, see \eqref{varphi}, starting from an independent Poisson number of individuals with parameter $\lambda x$.  Then, as $\lambda$ go to $\infty$,  we have, 
\begin{equation}\label{eq:conv}\left(\frac{1}{\lambda}L^{\lambda}_t(x),t\geq 0\right)\overset{\mathbbm{D}}{\Longrightarrow} (X_t(x),t\geq 0).\end{equation}
\end{theo}
Theorem \ref{thm:convergence} has been  established in \cite{zbMATH07107460} in the setting of non-immortal and non-explosive CSBPs, i.e. $\rho<\infty$ and $\int_{0}\frac{\ddr x}{|\psi(x)|}=\infty$ respectively. The main argument in \cite{zbMATH07107460} relies on a work of Helland \cite{helland1978continuity} about random time-change transformations. We will follow here a different approach  allowing us to treat all CSBPs. 
\vspace*{1mm}

The proofs of Proposition \ref{prop:mechanismskeleton} and Theorem \ref{thm:convergence} are given in Section \ref{sec:proofsonskeletons}.

\begin{rem} The fact that we considered the continuous-time Galton-Watson processes $L^{\lambda}$ with a Poisson number of initial individuals simplifies many calculations. Note that by Theorem \ref{thm:main}, $Z_t^{\lambda}$ has a Poisson law with parameter $\lambda X_t$ where $X$ is a CSBP$(\psi)$. The convergence of the one-dimensional law can therefore be seen as an application of the weak law of large numbers. The convergence in the finite-dimensional sense stays  true if one considers instead the processes $(L_t^{\lambda}/\lambda,t\geq 0)$ with $L_0^{\lambda}=[\lambda x]$. We refer the reader to \cite[pages 20-21]{LeGalllectures} for calculations in this direction.
\end{rem}
\section{Proofs of Theorem \ref{thm:branchingbitype} and Proposition \ref{Fellerprop}}\label{sec:constructionbitype}
In this section, we investigate the class of branching processes, potentially explosive, comprising one continuous and one discrete component. The construction is made in three steps. First, we establish the existence of a càdlàg solution to the martingale problem $\mathrm{MP}_{\mathbf{X}}(\mathscr{L},\mathcal{D})$ (at this stage the process is not yet known to be markovian). Second, we observe a duality relationship between the operator $\mathscr{L}$ and a system of o.d.es through the functions indexed by $(q,r)\in (0,\infty)\times (0,1)$ and defined by
\begin{equation}\label{eq:fqr}f_{q,r}(x,\ell):=e^{-qx}r^{\ell}, \text{ for all } (x,\ell) \in E=\mathbb{R}_+\times \mathbb{Z}.
\end{equation}
We show existence and uniqueness of a global solution to this system. 
Last, standard results, see for instance Ethier and Kurtz's book \cite[Chapter 4]{EK}, will apply and ensure that the duality holds true at the level of the one-dimensional laws, entailing uniqueness of the solution to the martingale problem and consequently the strong Markov property. Recall $\Delta=\{(x,\ell):x+\ell=\infty\}$ the cemetery point.
\begin{lem}\label{lem:existenceMP}  There exists a c\`adl\`ag solution $(\mathbf{X}_t)_{t\geq 0}=(X_t,L_t)_{t\geq 0}$ to the martingale problem associated to $(\mathscr{L},\mathcal{D})$, taking values in $E\cup \{\Delta\}$ and killed at the first explosion time, $\zeta:=\inf\{t>0: (X_{t-},L_{t-})=\Delta \text{ or } (X_t,L_t)=\Delta\}$. 
\end{lem}
\begin{proof} We apply here \cite[Theorem 5.4 page 199]{EK}. Recall the space of functions $\mathcal{D}$ defined in \eqref{eq:domainofgeneratorL}. Notice that by definition $C_0(\mathbb{R}_+\times \mathbb{Z})=\{f \text{ continuous s.t. } f(x,\ell)\rightarrow 0, \text{ as } (x,\ell)\rightarrow \Delta\}$. We must first check that $\mathscr{L}$ is a linear operator on $C_0(\mathbb{R}_+\times \mathbb{Z})$, see \cite[Page 8]{EK}. Plainly $\mathcal{D}\subset C_0(\mathbb{R}_+\times \mathbb{Z})$ and we only have to verify that the range of $\mathscr{L}$, i.e. $\{\mathscr{L}f: f\in \mathcal{D}\}$, is a subset of $C_0(\mathbb{R}_+\times \mathbb{Z})$. Recall $\mathscr{L}$ in \eqref{eq:genL}. We focus on the term of small jumps, the others are treated along similar arguments:
\begin{align*}
    &\Bigg\lvert x\sum_{k\geq 0}\int_0^1\left(f(x+y,\ell+k)-f(x,\ell)-yf'(x,\ell)\right)\pi(\ddr y,\ddr k)\Bigg\rvert \\
    &\leq  \int_0^1\!\!x\big\lvert f(x+y,\ell)-f(x,\ell)-yf'(x,\ell)\big\lvert\pi(\ddr y,\{0\})\\
    &\qquad +\sum_{k\geq 1}\int_0^1 x\big\lvert f(x+y,\ell+k)-f(x,\ell)-yf'(x,\ell)\big\lvert\pi(\ddr y,\ddr k).
\end{align*}
Since $f\in \mathcal{D}$, the integrand in the first integral above is continuous in $x$ and tends to zero as $x+\ell$ goes to $\infty$. It is moreover dominated by $Cy^2/2$ with $C:=\sup_{(x,\ell)\in E}|f''(x,\ell)|$, which is integrable on $(0,1)$ with respect to the measure $\pi(\ddr y,\{0\})$, see \eqref{eq:integrabilitypirho}. Hence, by Lebesgue theorem, the first integral term is continuous in $x$ and vanishes. This is also true for the second integral term, since the measure $\pi(\ddr y, \ddr k)$, restricted to $\mathbb{R}_+\times \mathbb{N}$, is finite, see \eqref{eq:integrabilitypirho}, and the integrand is continuous vanishing as $(x,\ell) \rightarrow \Delta$. We conclude that $\mathscr{L}f\in C_0(\mathbb{R}_+\times \mathbb{Z}_+)$.

We  argue now that $\mathcal{D}$ is dense in $C_0(\mathbb{R}_+\times \mathbb{Z}_+)$ for the uniform norm. The linear span of the functions $f_{q,r}$, $D=\mathrm{Vect}\left\{f_{q,r}, q>0,r<1\right\}$ is a subset of $\mathcal{D}$ and forms a subalgebra of $C_0(\mathbb{R}_+\times \mathbb{Z}_+)$, that is separating $\mathbb{R}_+\times \mathbb{Z}_+$. Therefore by Stone-Weierstrass theorem, see e.g. \cite[Theorem 7.32 page 162]{zbMATH03539473}, the latter is dense in $C_0(\mathbb{R}_+\times \mathbb{Z}_+)$.  Next, we show the positive maximum principle, see \cite[page 165]{EK}, namely  we check that for any $f\in \mathcal{D}$,  if $(x_0,\ell_0)\in E=\mathbb{R}_+\times \mathbb{Z}_+$ is such that $\sup f(x,\ell)=f(x_0,\ell_0)\geq 0$, then \[\mathscr{L}f(x_0,\ell_0)\leq 0.\]
First, plainly since $\kappa\geq 0$ and $\mathrm{k}\geq 0$, one has
$-\kappa x_0 f(x_0,\ell_0)\leq  0$  and $-\mathrm{k} \ell_0 f(x_0,\ell_0)\leq  0$.
Next, if $\ell_0\in \mathbb{Z}_+$ and $x_0\in (0,\infty)$  then $f(\ell_0+k,x_0)\leq f(\ell_0,x_0)$, similarly $f(\ell_0+k,x_0+y)\leq f(\ell_0,x_0)$ and necessarily $f'(x_0,\ell_0)=0$, so that the compensation term and the drift term vanish. Last, if $x_0=0$, then necessarily $f'(0,\ell_0)\leq 0$ and since $b\geq 0$, the drift term is also non-positive. In both cases $x_0>0$ and $x_0=0$, the second derivative term in $x_0$, coming from the diffusive part, is nonpositive since $x_0$ is a maximum and $\sigma\geq 0$. 
\end{proof}
The next lemma states an \textit{algebraic} duality relationship for the operator $\mathscr{L}$.
\begin{lem} 
For all $(x,\ell) \in \mathbb{R}_+\times \mathbb{Z}_+$, define on $(0,\infty)\times (0,1)$, the map 
\begin{center}
$g_{x,\ell}:(q,r)\mapsto e^{-qx}r^{\ell}$.
\end{center}
One has
\begin{align}\mathscr{L}f_{q,r}(x,\ell)&=xe^{-qx}r^{\ell}\Psi_c(q,r)+e^{-qx}\ell r^{\ell-1}\Psi_d(q,r)\label{eq:Lfqr}\\
&=-\Psi_c(q,r)\frac{\partial}{\partial q}g_{x,\ell}(q,r)+\Psi_d(q,r)\frac{\partial}{\partial r}g_{x,\ell}(q,r)\label{eq:Lfqrwithderivatives}\\
&=-\mathbf{\Psi}\cdot \nabla g_{x,\ell}(q,r) \label{eq:algebraicduality}
\end{align}
where $\mathbf{\Psi}=\big(\Psi_c, -\Psi_d\big)$. 
\end{lem}
\begin{proof}
Recall $\mathscr{L}$ in \eqref{eq:genL}, one has 
\begin{align*}
\mathscr{L}f_{q,r}(x,\ell)=&-\gamma x qe^{-qx}r^{\ell}-b \ell qe^{-qx}r^{\ell}+\frac{\sigma^2}{2}xq^2e^{-qx}r^{\ell}-\kappa xe^{-qx}r^{\ell}-\mathrm{k} \ell e^{-qx}r^{\ell}\\
&+x\sum_{k\geq 0}\int_{\mathbb{R}_+}\left(e^{-q(x+y)}r^{\ell+k}-e^{-qx}r^{\ell}+y\ind_{(0,1)}(y)qe^{-qx}r^{\ell}\right)\pi(\ddr y,\ddr k)\nonumber\\
&+\ell\sum_{k\geq 0}\int_{\mathbb{R}_+}\left(e^{-q(x+y)}r^{\ell+k}-e^{-qx}r^{\ell}\right)\rho(\ddr y,\ddr k)+\underbrace{d\ell  \big( e^{-qx}r^{\ell-1}-e^{-qx}r^{\ell}\big)}_{=d\ell r^{\ell-1}e^{-qx}(1-r)}.\nonumber
\end{align*}
By rearranging everything in order to make appear $\Psi_c$ and $\Psi_d$, as defined in \eqref{eq:Psic} and \eqref{eq:Psid},
we get
$$\mathscr{L}f_{q,r}(x,\ell)=xe^{-qx}r^{\ell}\Psi_c(q,r)+e^{-qx}\ell r^{\ell-1}\Psi_d(q,r).$$
Using the facts that 
\begin{center}
$xe^{-qx}r^{\ell}=-\frac{\partial}{\partial q}g_{x,\ell}(q,r)$ and $\ell e^{-qx}r^{\ell-1}=\frac{\partial}{\partial r}g_{x,\ell}(q,r)$, 
\end{center}
we have finally 
\[\mathscr{L}f_{q,r}(x,\ell)=-\Psi_c(q,r)\frac{\partial}{\partial q}g_{x,\ell}(q,r)+\Psi_d(q,r)\frac{\partial}{\partial r}g_{x,\ell}(q,r).\]
\end{proof}
\begin{lem}\label{lem:uniquenessode}
There exists a unique solution $t\mapsto F_t(q,r):=\big(u_t(q,r),f_t(q,r)\big)$ to 
\begin{align}\label{odeuf2}
&\frac{\ddr }{\ddr t}
u_t(q,r)
=-\Psi_c\big(u_t(q,r),f_t(q,r)\big),\ \frac{\ddr }{\ddr t}
f_t(q,r)
=\Psi_d\big(u_t(q,r),f_t(q,r)\big)\\
&u_0(q,r)=q, \ f_0(q,r)=r.
\end{align}
Furthermore, for all $t\in [0,\infty)$, $r\in (0,1)$ and $q\in (0,\infty)$,
\begin{center} 
$0<u_t(q,r)<\infty$ and $0<f_t(q,r)<1$.
\end{center}
\end{lem}
\begin{proof}
All first partial derivatives of the function $\mathbf{\Psi}(q,r)=\left(\Psi_c(q,r),
-\Psi_d(q,r)\right)$ are continuous and bounded on domains of the form $(a,b)\times (l,r)\subset (0,\infty)\times (0,1)$. This entails that $\mathbf{\Psi}$ is locally Lipschitz on $(0,\infty)\times (0,1)$, see e.g. \cite[Chapter 6, page 174]{zbMATH03587373}. Cauchy-Lipschitz theorem ensures the existence of a local solution. We now show that it is bounded below and above by positive and finite quantities.  

We first find an upper bound for $\Psi_c$. Rewrite the integrand of
\eqref{eq:Psic} with the help of the identity
\begin{equation}\label{identityintegrand}e^{-qy}r^{k}-1+qy\ind_{(0,1)}(y)= e^{-qy}r^{k}-e^{-qy}+e^{-qy}-1+qy\ind_{(0,1)}(y),\ \forall y\in (0,\infty),k\geq 0.
\end{equation}
Since
$e^{-qy}r^{k}-e^{-qy}\leq 0$, $e^{-qy}-1\leq 0$ 
and there is $C>0$ such that for all $y\in (0,1)$, $0\leq e^{-qy}-1+qy\leq C\frac{(qy)^2}{2}$, we get
\begin{align}
\Psi_c(q,r)&\leq \sum_{k\in \mathbb{N}}\int_{(0,\infty)}\left(e^{-qy}-1+qy\ind_{(0,1)}(y)\right) \pi(\ddr y, \ddr k)-\gamma q+\frac{\sigma^2}{2}q^2-\kappa\nonumber\\
&\leq \sum_{k\in \mathbb{N}}\int_{(0,1)}\left(e^{-qy}-1+qy\right) \pi(\ddr y, \ddr k)-\gamma q+\frac{\sigma^2}{2}q^2 \nonumber\\
&\leq c_{\sigma,\pi}q^2-\gamma q
\end{align}
with $c_{\sigma,\pi}:=\frac{\sigma^2}{2}+C\int_{0}^{1}y^2\pi(\ddr y,\mathbb{Z}_+)$.
By comparison, we see from the o.d.e
\[\frac{\ddr }{\ddr t}u_t(q,r)=-\Psi_c\big(u_t(q,r),f_t(q,r)\big)\]
that \[u_t(q,r)\geq \frac{q\gamma e^{-\gamma t}}{\gamma +\frac{qc_{\sigma,\pi}}{2}\big(e^{-\gamma t}-1\big)}>0.\]

We now show that $f_t(q,r)<1$. Recall \eqref{eq:Psid}. Plainly $e^{-qy}r^{k+1}-r=r\big(e^{-qy}r^{k}-1\big)\leq 0$ and $b,d\geq 0$, thus $\Psi_d(q,r)\leq d(1-r)$. Hence,
$$\frac{\ddr }{\ddr t}f_t(q,r)=\Psi_d(u_t(q,r),f_t(q,r))\leq d\big(1-f_t(q,r)\big).$$
By setting $k_t(q,r)=1-f_t(q,r)$, we get
$\frac{\ddr }{\ddr t}k_t(q,r)\geq -dk_t(q,r),$ so $k_t(q,r)\geq re^{-dt}>0$ and $f_t(q,r)<1$ for all $t\geq 0$.
\\ 

It remains to establish that for all $(q,r)\in (0,\infty)\times (0,1)$,
$$u_t(q,r)<\infty \text{ and } f_t(q,r)>0,  \forall t\geq 0.$$ 
Recall
\[\frac{\ddr }{\ddr t}u_t(q,r)=-\Psi_c\big(u_t(q,r),f_t(q,r)\big)\]
with
\begin{align*}
-\Psi_c\big(q,r\big) & =  - \sum_{k\geq 0}\int_{\mathbb{R}_+} \left(e^{-qy}r^{k}-1+qy\ind_{(0,1)}(y)\right)\pi(\ddr y, \ddr k)+\gamma q- \frac{\sigma^2}{2}q^2+\kappa \\
& \leq   \int_{\mathbb{R}_+} \left(1-e^{-qy}- qy \ind_{(0,1)}(y)\right)\pi(\ddr y, \{0\}) + \sum_{k\geq 1}\int_{\mathbb{R}_+}  \pi(\ddr y, \ddr k)+ \gamma q +\kappa\\
& \leq  \int_{1}^{\infty}  \pi(\ddr y, \{0\}) + \sum_{k\geq 1}\int_{\mathbb{R}_+}  \pi(\ddr y, \ddr k)+ \gamma q +\kappa=:c_{\pi,\kappa} +\gamma q  \\
\end{align*}
hence $\frac{\ddr }{\ddr t}u_t(q,r) \leq c_{\pi,\kappa} + \gamma u_t(q,r)$, which in turn implies, recalling $u_0(q,r)=q$,
\begin{align}
\label{bound:utqr}
u_t(q,r) \leq   \frac{(c_{\pi,\kappa}+\gamma q) e^{\gamma t}- c_{\pi,\kappa}}{\gamma}.
\end{align}

For any $\eta \in (0,\infty)$, define $u_t^{d}(q,\eta)$ such that 
\begin{equation}\label{def:utd}
e^{-u_t^{d}(q,\eta)}=f_t(q,e^{-\eta}).
\end{equation} 
Set \[\bar{\Psi}_d(q,\eta):=e^{\eta}\Psi_d(q,e^{-\eta}).\]
We see from the o.d.e  \eqref{odeuf2} solved by $f_t(q,r)$ that
\begin{align*}
\frac{\ddr }{\ddr t}u_t^{d}(q,\eta)=-\bar{\Psi}_d\big(u_t(q,e^{-\eta}),u_t^{d}(q,\eta)\big).
\end{align*}
By plugging $r=e^{-\eta}$ in the expression \eqref{eq:Psid} of $\Psi_d(q,r)$, we get
\begin{align}
-\bar{\Psi}_d\big(q,\eta \big) 
& = \sum_{k=0}^{\infty}\int_{\mathbb{R}_+}\left(1-e^{-qy} e^{-\eta k}\right)\rho(\ddr y,\ddr k)+ bq-d (e^\eta-1)-\mathrm{k} \label{eq:barPsid}\\
& \leq \int_{\mathbb{R}_+}\left(1-e^{-qy} \right)\rho(\ddr y,\{0\}) + \sum_{k=1}^{\infty}\int_{\mathbb{R}_+} \rho(\ddr y,\{k\})  + bq \nonumber\\
& \leq \left(\int_{0}^1 y \rho(\ddr y,\{0\}) +b\right) q + \int_{1}^{\infty} \rho(\ddr y,\{0\}) + \sum_{k=1}^{\infty}\int_{\mathbb{R}_+} \rho(\ddr y,\{k\})  =: c_{\rho,1} q +  c_{\rho,2}, \nonumber
\end{align}
hence $\frac{\ddr }{\ddr t} u_t^{d}(q,\eta) = c_{\rho,1} u_t(q,e^{-\eta})  +  c_{\rho,2}$, and given the bound \eqref{bound:utqr} above on $u_t(q,r)$, we get
\[u_t^{d}(q,\eta) \leq \eta  +  c_{\rho,1} \frac{(c_{\pi,\kappa}+\gamma q) e^{\gamma t}- \gamma c_{\pi,\kappa} t }{\gamma^2} +  c_{\rho,2} t,\] which ensures $u_t^{d}(q,\eta) <\infty$, and in turn, by \eqref{def:utd},
$f_t(q,r) >0$.\\

Let $(q,r)\in (0,\infty)\times (0,1)$. We have just established that up to any time $t\geq 0$, any solution $[0,t]\ni s\mapsto (u_s(q,r),f_s(q,r))$ to \eqref{odeuf2} stays in a domain $D\subset (0,\infty)\times (0,1)$. The function $\mathbf{\Psi}$ being Lipschitz on such domain $D$, there is a unique solution to the equation up to time $t$. The latter being arbitrary, the solution is global, namely it is defined on the whole half-line. 

\end{proof}
We now characterize the one-dimensional law of the solution $(X,L)$ with the help of $t\mapsto F_t(q,r)$.
\begin{lem}  
For all $(q,r)\in (0,\infty)\times [0,1)$, 
\begin{equation}\label{eq:dualsemigroupinlem}
\mathbb{E}_{(x,n)}[e^{-qX_t}r^{L_t}]=e^{-x u_t(q,r)}f_t(q,r)^n \qquad \forall (x,n)\in \mathbb{R}_+\times\mathbb{Z}_+, \forall t\geq 0.
\end{equation}
\end{lem}
\begin{proof}
The task here is to establish the duality relationship between the solution to MP$(\mathscr{L},\mathcal{D})$, $(X,L)$ provided by Lemma \ref{lem:existenceMP}, and the deterministic process $t\mapsto F_t(q,r)=(u_t(q,r),f_t(q,r))$. Having noticed the algebraic relationship \eqref{eq:algebraicduality}, it remains to apply \cite[Theorem 4.11 page 192]{EK}, and thus to verify its condition (4.50). The latter shrinks here to the following. Let \[g\big((x,\ell),(u,f)\big):=x\Psi_c(u,f)e^{-ux}f^{\ell}+e^{-ux}\ell f^{\ell-1}\Psi_d(u,f),\]
one has to check that for any $T>0$,
\[S_T:=\sup_{s,t\leq T}\left\lvert g\big((X_s,L_s),(u_t,f_t)\big)\right\lvert,\]
is integrable, with $(u_t,f_t)=F_t(q,r)$. Using the inequalities \[xe^{-ux}\leq 1/u \text{ and } \ell f^{\ell}=\ell e^{-\ell \ln(1/f)}\leq 1/\ln(1/f),\]
we get for all $(x,\ell)\in \mathbb{R}_+\times \mathbb{Z}_+$, all $q\in (0,\infty),r\in (0,1)$ and  all $t\geq 0$
\[\left\lvert g\big((x,\ell),(u_t,f_t)\big)\right \lvert \leq \left\lvert\frac{\Psi_c(u_t,f_t)}{u_t} \right\lvert+\left\vert\frac{\Psi_d(u_t,f_t)}{f_t\ln(1/f_t)}\right\lvert.\]
According to Lemma \ref{lem:uniquenessode}, $u_t\in (0,\infty)$ and $f_t\in (0,1)$ for all $t$. By continuity, their extrema on the compact time interval $[0,T]$ belongs to $(0,\infty)$ and $(0,1)$, therefore $S_T$ is bounded.
\end{proof}
\begin{proof}[Proof of Theorem \ref{thm:branchingbitype}] It follows by applying \cite[Theorem 4.2 page 184]{EK}.\end{proof}

We establish now Proposition \ref{Fellerprop} and verify among other things the branching property of $\mathbf{X}$.
\begin{proof}[Proof of Proposition \ref{Fellerprop}] 
We start by verifying the Feller property and exhibiting a core.  
Recall $D$ the space generated by the linear combinations of the functions $f_{q,r}$. Since as $t$ goes to $0$, $u_t(q,r)\rightarrow q$ and $f_t(q,r)\rightarrow r$, one has for any $(x,\ell)\in \mathbb{R}_+\times \mathbb{Z}_+$, $P_tf_{q,r}(x,\ell)\underset{t\rightarrow 0}{\longrightarrow} f_{q,r}(x,\ell)$. Furthermore, $D\subset \mathcal{D}$ and $P_tD\subset C_0(\mathbb{R}_+\times \mathbb{Z}_+)$. Therefore, by density  of $D$ in $C_0(\mathbb{R}_+\times \mathbb{Z}_+)$, we see that for all $f\in C_0(\mathbb{R}_+\times \mathbb{Z}_+)$, $$P_tf\in C_0(\mathbb{R}_+\times \mathbb{Z}_+) \text{ and } P_tf(x,\ell)\underset{t\rightarrow 0}{\longrightarrow} f(x,\ell).$$
Last, since $D$ is a subset of the domain of $\mathscr{L}$, dense in $C_0(\mathbb{R}_+\times \mathbb{Z}_+)$ and $P_tD\subset D$, this is a core, see e.g. \cite[Proposition 19.9]{Kallenberg}. The branching property \eqref{branchingprop} readily follows from \eqref{dualsemigroup}, indeed
\begin{align*}
\mathbb{E}_{(x+y,n+m)}\big[e^{-qX_t}r^{L_t}\big]&=e^{-xu_t(q,r)}f_t(q,r)^{n}e^{-yu_t(q,r)}f_t(q,r)^{m}\\
&=\mathbb{E}_{(x,n)}\big[e^{-qX_t}r^{L_t}\big]\mathbb{E}_{(y,m)}\big[e^{-qX_t}r^{L_t}\big].
\end{align*}
\end{proof}
We now establish Proposition \ref{prop:autonomous} where conditions are given for the coordinates of the bi-type branching process $(X,L)$ to be autonomous.
\begin{proof}[Proof of Proposition \ref{prop:autonomous}]
We only give the proof in case 1, since the two proofs are the same. We thus assume that $\pi(\ddr y,\ddr k)= \nu(\ddr y) \delta_0(\ddr k)$ for a measure  $\nu(\ddr y)$ on $(0,\infty)$ such that $\int_0^\infty (1\wedge y^2) \nu(\ddr x)<\infty$. In this case, the map
 $(q,r) \mapsto \Psi_c(q,r)$ does not depend on the variable $r$ and  can therefore be simply denoted by $q\mapsto \Psi_c(q,1)$. By \eqref{odeuf}, the map $t \mapsto u_t(0+,r)$ is solution of the o.d.e
$$\frac{\ddr }{\ddr t}u_t = -\Psi_c(u_t,1), \ u_0=0.$$
The condition $\int_{0}\frac{\ddr x}{|\Psi_c(x,1)|}=\infty$ is necessary and sufficient for this o.d.e. to have no solution other than the null function. Hence under this condition, $u_t(0+,r)=0$ for all $t\geq 0$  and $(L_t)$ then satisfies for any $n \in \Z^+$,
$$
\mathbb{E}_{n}[r^{L_t}]=f_t(0+,r)^n,
$$
where $f_t(0+,r)$ is the unique solution to the o.d.e:
$$
\frac{\ddr }{\ddr t} f_t(0+,r)
=-\Psi_d \big(0+,f_t(0+,r)\big),\ f_0(0+,r)=r,$$ meaning that $(L_t)$ is a discrete branching process with branching mechanism
$$\varphi:[0,1)\ni r\to \Psi_d(0+,r)=\sum_{k \geq 1} (r^{k+1}-r) \rho(\R_+,\{k\}) +d(1-r)-\mathrm{k}r.$$ 
\end{proof}

\section{Proof of Theorem \ref{thm1:main}}\label{sec:proofmainthm1}
Let $\psi$ be a branching mechanism, see \eqref{eq:branchingmechanismpsi}.
Recall the infinitesimal generator $\mathcal{G}$ of the CSBP$(\psi)$, see \eqref{eq:G}. We will establish in this section the algebraic intertwining relationship \eqref{eq:intertwiningHKG} between $\mathcal{G}$ and $\mathcal{H}$.
We start by identifying the triplet of the Esscher transform of $\psi$, $\psi_\lambda(\cdot)=\psi(\lambda+\cdot)-\psi(\lambda)$, for any $\lambda\geq 0$, see \eqref{eq:esscher}.
\begin{lem} For any $\lambda\in [0,\infty)$,
\[\psi_\lambda(q)=\frac{\sigma^2}{2} q^2 +\int_{(0,\infty)}   \nu(\ddr y) \expp{-\lambda y} [\expp{-q y}- 1 + q y]+\psi'(\lambda)q,\]
with 
\[\psi'(\lambda)=\sigma^2 \lambda-\gamma+\int_{(0,1)}(1-e^{-\lambda y})y\nu(\ddr y)-\int_{1}^{\infty}\nu(\ddr y)ye^{-\lambda y}.\]
\end{lem}
\begin{proof}
\begin{align}
\label{eq:mecha} \psi_{\lambda}(q)&:=\psi(\lambda+q)- \psi(\lambda)\\
 &= \sigma^2(\lambda+q)^2/2 - \gamma (\lambda+q) - (\sigma^2\lambda^2/2  - \gamma \lambda\big) \nonumber
 \\
 &\qquad +\int_{(0,\infty)}  \nu(\ddr y) [\expp{-(q+\lambda) y}- 1 + (q+\lambda) y \ind_{(0,1)}(y)] 
\\
&= \frac{\sigma^2}{2} q^2 +\int_{(0,\infty)}   \nu(\ddr y) \expp{-\lambda y} [\expp{-q y}- 1 + q y]+ (\lambda \sigma^2-\gamma) q \\
&\qquad+\left(\int_{(0,\infty)} \nu(\ddr y) y \ind_{(0,1)}(y) \; (1-\expp{-\lambda y})-\int_{1}^{\infty}\!\!ye^{-\lambda y}\nu(\ddr y)\right)q,\\
&=\frac{\sigma^2}{2} q^2 +\int_{(0,\infty)}   \nu(\ddr y) \expp{-\lambda y} [\expp{-q y}- 1 + q y]+\psi'(\lambda)q.
\end{align}
\end{proof}
Recall the Poisson kernel $K$ with parameter $\lambda$ in \eqref{Poissonkernel} and the space $\mathcal{D}$, see \eqref{eq:domainofgeneratorL}. One can\footnote{see the forthcoming Equations \eqref{eq:Kf'}-\eqref{eq:Kf''} for the calculation of the two first derivatives.} easily check that for any $f\in \mathcal{D}$, $Kf\in D_c$, see \eqref{Dc}.
\\

Recall the operator $\mathcal{H}$ in \eqref{Hi}-\eqref{Hii} and define
\begin{equation}\label{R}
\mathcal{R}f(x, \ell):= \begin{cases}
\frac{\psi(\lambda)}{\lambda}\ell\left[f(x,\ell-1)-f(x,\ell)\right]&, \text{ if } \lambda\geq \rho,\\
-\psi(\lambda)x\left[f(x,\ell+1)-f(x,\ell)\right]&, \text{ if } \lambda\leq \rho.
\end{cases}
\end{equation}

Observe that the factor term in $\mathcal{R}$ lying in front of the incremental term, $f(x,\ell-1)-f(x,\ell)$ or $f(x,\ell+1)-f(x,\ell)$, is always non-negative. This will be important when interpreting those terms as jump rates. At this stage of the study, at which an algebraic relation is targeted, this plays however essentially no role. Indeed 
the following lemma shows that after Poissonization with parameter $\lambda x$, the form of the operator $\mathcal{R}$ does not depend on whether $\lambda\leq \rho$ or $\lambda>\rho$. This will allow us to study both cases simultaneously.
\begin{lem}\label{lem:KRf} For any $\lambda\in (0,\infty)$, 
\begin{align*}
K\mathcal{R}f(x)&=\sum_{\ell \geq 0}\expp{-\lambda x}\frac{(\lambda x)^{\ell}}{\ell!} \big(-\psi(\lambda)x\big)\left[f(x,\ell+1)-f(x,\ell)\right]\\
&=\sum_{\ell \geq 1}\expp{-\lambda x}\frac{(\lambda x)^{\ell}}{\ell!}\frac{\psi(\lambda)}{\lambda} \ell\left[f(x,\ell-1)-f(x,\ell)\right]. 
\end{align*}
\end{lem}
\begin{proof}[Proof of Lemma \ref{lem:KRf}]
\begin{align*}
&\sum_{\ell \geq 0}\expp{-\lambda x}\frac{(\lambda x)^{\ell}}{\ell!} \left(-x\psi(\lambda)\right)\left[f(x,\ell+1)-f(x,\ell)\right]\\
&=x\psi(\lambda)\sum_{\ell \geq 0}\expp{-\lambda x}\frac{(\lambda x)^{\ell}}{\ell!} \left[f(x,\ell)-f(x,\ell+1)\right]\\
&=\psi(\lambda)\sum_{\ell \geq 0}\frac{1}{\lambda}\expp{-\lambda x}\frac{(\lambda x)^{\ell+1}}{(\ell+1)!}(\ell+1) \left[f(x,\ell)-f(x,\ell+1)\right]\\
&=\frac{\psi(\lambda)}{\lambda}\sum_{\ell \geq 1}\expp{-\lambda x}\frac{(\lambda x)^{\ell}}{\ell!}\ell \left[f(x,\ell-1)-f(x,\ell)\right].
\end{align*}
\end{proof}

By linearity of the operator $\mathcal{G}$, we can separate the study into its local and  non-local parts. We start therefore with the generator of the Feller diffusion.

\subsection{The case of the Feller diffusion $(\Psi(q)=\frac{\sigma^2}{2}q^2-\gamma q)$}
We focus here on the setting of a pure diffusive CSBP process. The integral term in $\mathcal{G}$ vanishes and the branching mechanism shrinks to $\psi(q)=\frac{\sigma^2}{2}q^2-\gamma q$ with $\gamma\in \mathbb{R}$. Notice that $\frac{\psi(\lambda)}{\lambda}=\frac{\sigma^2}{2}\lambda-\gamma$, $\psi'(\lambda)=\sigma^2\lambda-\gamma$ and $\rho=\frac{2\gamma}{\sigma^2}\vee 0$. One has $$\mathcal{G}^{\psi_\lambda} f(x, \ell)=\mathcal{G}^{\psi_\lambda} f_{\ell}(x)= x \, \left[ \frac{\sigma^2}{2} f''(x, \ell)-(\sigma^2 \lambda - \gamma) f'(x,\ell)\right].$$
Introduce the operator $\mathcal{J}^{L}$ (the superscript $L$ is for \textit{local}): \begin{equation}\label{eq:Hrho}
\mathcal{J}^{L} f(x, \ell):=  \sigma^2 \ell  f'(x,\ell) + \ell \frac{\sigma^2}{2} \lambda \, \left[f(x,\ell+1)-f(x,\ell)\right].
\end{equation}
One has
\[\mathcal{H}f(x, \ell)= 
\mathcal{G}^{\psi_\lambda} f(x, \ell)+\mathcal{J}^{L} f(x, \ell)+\mathcal{R}f(x, \ell).\]

We now establish that $\mathcal{H}$ intertwines $\mathcal{G}$, namely we show that the following relationship holds:
\[\mathcal{G} K f(x)=K\mathcal{H}f(x).\]
To prove this, we shall expand the following expression
$$\mathcal{G} K f(x)= x \left(\frac{\sigma^2}{2} (Kf)''(x)+ \gamma (Kf)'(x)\right)$$
and compare it with
\begin{align}
K \mathcal{H}f(x) &=  \sum_{\ell \geq 0} \expp{-\lambda x} \frac{(\lambda x)^\ell}{\ell!}
\Big[\underbrace{ x \Big(\frac{\sigma^2}{2} f''(x, \ell)- (\sigma^2 \lambda - \gamma) f'(x,\ell)\Big)}_{G^{\psi_\lambda}f(x,\ell)}\Big ] \label{A}\\
& + \sum_{\ell \geq 0} \expp{-\lambda x}  \frac{(\lambda x)^\ell}{\ell!} \left(\underbrace{\ell \sigma^2 f'(x,\ell) + \ell \frac{\sigma^2}{2} \lambda \left[f(x,\ell+1)-f(x,\ell)\right]}_{\mathcal{J}^{L}f(x,\ell)}\right) \label{B}\\
&+ \underbrace{\sum_{\ell \geq 1}\expp{-\lambda x}\frac{(\lambda x)^{\ell}}{\ell!}\left[\ell\left(\frac{\sigma^2}{2} \lambda-\gamma\right)\right] \big(f(x,\ell-1)-f(x,\ell)\big)}_{K\mathcal{R}f(x)} \label{C}. 
\end{align}
Expanding $\mathcal{G} K f(x)$ requires to compute the first two derivatives of the expression $(Kf)(x)$:
\begin{align}
(Kf)'(x)=& \ \lambda \sum_{\ell \geq 0} \expp{-\lambda x} \left[ \frac{(\lambda x)^{\ell-1}}{(\ell-1)!} -\frac{(\lambda x)^{\ell}}{\ell!} \right] f(x,\ell) +
\sum_{\ell \geq 0} \expp{-\lambda x} \frac{(\lambda x)^\ell}{\ell !}  f'(x,\ell). \label{eq:Kf'}\\ 
(Kf)''(x) =& 
\ \lambda^2 \sum_{\ell \geq 0} \expp{-\lambda x} \left[\frac{(\lambda x)^{\ell-2}}{(\ell-2)!} -2 \frac{(\lambda x)^{\ell-1}}{(\ell-1)!} 
+ \frac{(\lambda x)^{\ell}}{\ell!} \right] f(x,\ell) \nonumber \\ 
&+2 \lambda \sum_{\ell \geq 0} \expp{-\lambda x} \left[ \frac{(\lambda x)^{\ell-1}}{(\ell-1)!} -\frac{(\lambda x)^{\ell}}{\ell!} \right]   f'(x,\ell) 
+ \sum_{\ell \geq 0} \expp{-\lambda x} \frac{(\lambda x)^\ell}{\ell !}  f''(x,\ell). \label{eq:Kf''}
\end{align}
Next, the expression for $\mathcal{G}Kf(x)$ can be splitted into three terms, that we distinguish according to the partial derivative of $f(x,\ell)$ (zeroth, first or second) that is used:\\

$\mathcal{G}K f(x)$
\begin{align}
& = x \frac{\sigma^2}{2}  \lambda^2 \sum_{\ell \geq 0} \expp{-\lambda x} \left[\left(\frac{(\lambda x)^{\ell-2}}{(\ell-2)!} -2 \frac{(\lambda x)^{\ell-1}}{(\ell-1)!} +   \frac{(\lambda x)^{\ell}}{\ell!} \right) 
+ \gamma \lambda  \left( \frac{(\lambda x)^{\ell-1}}{(\ell-1)!} -\frac{(\lambda x)^{\ell}}{\ell!} \right) \right] f(x,\ell)   \label{I}\\ 
& + x \sum_{\ell \geq 0} \expp{-\lambda x} \left[  \sigma^2  \lambda \left( \frac{(\lambda x)^{\ell-1}}{(\ell-1)!}-\frac{(\lambda x)^{\ell}}{\ell!} \right)   
+ \gamma  \frac{(\lambda x)^\ell}{\ell !}  \right] f'(x,\ell)  \label{II} \\
&+ x \left[ \frac{\sigma^2}{2} \sum_{\ell \geq 0} \expp{-\lambda x} \frac{(\lambda x)^\ell}{\ell !} \right] f''(x,\ell)\label{III}.
\end{align}

In \eqref{III}, the term in factor of $f''(x,\ell)$ identifies with the similar term in $K \mathcal{H}f(x)$ in \eqref{A} so there is nothing to do here.

To transform \reff{I}, the idea is to separate it into three sums  and  reindex them in a convenient manner. We rewrite \reff{I} as \\
$$x\sum_{\ell \geq 0} \expp{-\lambda x} \left[ \frac{\sigma^2}{2}  \lambda^2 \left(\frac{(\lambda x)^{\ell-2}}{(\ell-2)!} -\frac{(\lambda x)^{\ell-1}}{(\ell-1)!}\right) +  \frac{\sigma^2}{2}  \lambda^2 \left(\frac{(\lambda x)^{\ell}}{\ell!}-\frac{(\lambda x)^{\ell-1}}{(\ell-1)!} \right) 
+ \gamma \lambda  \left( \frac{(\lambda x)^{\ell-1}}{(\ell-1)!} -\frac{(\lambda x)^{\ell}}{\ell!} \right) \right] f(x,\ell).$$
Therefore
\begin{align*} 
&\reff{I}\\
&=x\sum_{\ell \geq 0} \expp{-\lambda x} \left[- \frac{\sigma^2}{2}  \lambda^2 \frac{(\lambda x)^{\ell}}{\ell!}\left[f(x,\ell+1)-f(x,\ell)\right]+
\frac{\sigma^2}{2}  \lambda^2 \frac{(\lambda x)^{\ell-1}}{(\ell-1)!}\left[f(x,\ell+1)-f(x,\ell)\right] \right.
\\
&\qquad \qquad \qquad \qquad \left. +\gamma \lambda \frac{(\lambda x)^{\ell}}{\ell!}\left[f(x,\ell+1)-f(x,\ell)\right]\right]
\\
&=x\sum_{\ell \geq 0} \expp{-\lambda x} \frac{(\lambda x)^{\ell}}{\ell!} \left(\gamma\lambda- \frac{\sigma^2}{2}  \lambda^2\right) \left[f(x,\ell+1)-f(x,\ell)\right]+x\frac{\sigma^2}{2}   \lambda^2 \sum_{\ell \geq 0}\expp{-\lambda x}\frac{(\lambda x)^{\ell-1}}{(\ell-1)!}\ell\left[f(x,\ell+1)-f(x,\ell)\right].\\
&=\underbrace{x\sum_{\ell \geq 0} \expp{-\lambda x} \frac{(\lambda x)^{\ell}}{\ell!} \left(\gamma\lambda- \frac{\sigma^2}{2}   \lambda^2\right) \left[f(x,\ell+1)-f(x,\ell)\right]}_{\mathcal{K}\mathcal{R}f(x)}+ \sum_{\ell \geq 0}\expp{-\lambda x}\frac{(\lambda x)^{\ell}}{\ell!}\frac{\sigma^2}{2}   \lambda \ell\left[f(x,\ell+1)-f(x,\ell)\right].\\
\end{align*} 
In the same way:
\begin{align*}
\reff{II}&=\sum_{\ell \geq 0}\expp{-\lambda x}\frac{(\lambda x)^{\ell}}{\ell!}\left[\sigma^{2} \lambda x \left(\frac{\ell}{\lambda x}-\ell\right)+\gamma x\right]f'(x,\ell)\\
&=\sum_{\ell \geq 0}\expp{-\lambda x}\frac{(\lambda x)^{\ell}}{\ell!}\sigma^{2}\ell f'(x,\ell) +\sum_{\ell \geq 0}\expp{-\lambda x}\frac{(\lambda x)^{\ell}}{\ell!}\big(\gamma-\sigma^{2}\lambda\big)xf'(x,\ell).
\end{align*}
Recall $\mathcal{J}^{L}$ and $\mathcal{G}^{\psi_\lambda}$. We get
\begin{align*}
\mathcal{G}Kf(x)=&\reff{I}+\reff{II}+\reff{III}\\
= &\mathcal{K}\mathcal{R}f(x)+
\overbrace{\sum_{\ell \geq 0}\expp{-\lambda x}\frac{(\lambda x)^{\ell}}{\ell!}\frac{\sigma^2}{2}   \lambda \ell\left[f(x,\ell+1)-f(x,\ell)\right]+\sum_{\ell \geq 0}\expp{-\lambda x}\frac{(\lambda x)^{\ell}}{\ell!}\sigma^{2}\ell f'(x,\ell)}^{\mathcal{K}\mathcal{J}^{L}f(x)}\\
&\qquad +\underbrace{\sum_{\ell \geq 0}\expp{-\lambda x}\frac{(\lambda x)^{\ell}}{\ell!}\big(\gamma-\sigma^{2}\lambda\big)xf'(x,\ell)+\sum_{\ell \geq 0}\expp{-\lambda x}\frac{(\lambda x)^{\ell}}{\ell!}\frac{\sigma^{2}}{2}xf''(x,\ell)}_{K\mathcal{G}^{\psi_\lambda}f(x)}\\
&=K\big(\mathcal{R}+\mathcal{J}^{L}+\mathcal{G}^{\psi_\lambda}\big)f(x)\\
&=K\mathcal{H}f(x).
\end{align*}
\subsection{The case of pure jump CSBPs} 
We now deal with CSBP processes with no local part and assume $\sigma=\gamma=0$. Define
\begin{equation}\label{JNL}
\mathcal{J}^{NL}f(x,\ell):=\ell\sum_{k \geq 0} \int_{(0,\infty)}   \nu(\ddr y) y \expp{-\lambda y} \frac{(\lambda y)^k}{(k+1)!} [f(x+y,\ell+k)-f(x,\ell)]. 
\end{equation}
Notice that $\mathcal{H}=\mathrm{c}+\mathcal{G}^{\psi_\lambda}+\mathcal{J}^{NL}+\mathcal{R}$, with $\mathrm{c}f(x,\ell):=-\kappa x f(x,\ell)$.
We are going to compute \begin{equation}\label{eq:GKf}
\mathcal{G}Kf(x)=  x\int_{(0,\infty)} \nu(\ddr y) [Kf(x+y)-Kf(x)- y \ind_{(0,1)}(y) (Kf)'(x)]-\kappa x Kf(x).
\end{equation}
Plainly, $-\kappa x Kf(x)=-\sum_{\ell=0}^{\infty}\frac{(\lambda x)^{\ell}}{\ell !}\kappa xf(x,\ell)= K(\mathrm{c}f)(x)$, so that the killing term in $\mathcal{G}$ matches with a killing term along the continuous component of $\mathcal{H}$. We now study the integrand in \eqref{eq:GKf}. Write the Poisson kernel evaluated at $x+y$:
$$ Kf(x+y)= \sum_{\ell \geq 0} \expp{-\lambda (x+y)} \frac{(\lambda (x+y))^\ell}{\ell !} f(x+y, \ell)$$
and its derivative evaluated at $x$:
$$ (Kf)'(x)= 
- \sum_{\ell \geq 0}  \lambda  \expp{-\lambda x} \frac{(\lambda x)^\ell}{\ell !} f(x, \ell) +
\sum_{\ell \geq 1} \lambda  \expp{-\lambda x} \frac{(\lambda x)^{\ell-1}}{(\ell-1)!} f(x, \ell) +
\sum_{\ell \geq 0}   \expp{-\lambda x} \frac{(\lambda x)^\ell}{\ell !} f'(x, \ell).
$$
The expression for $x[Kf(x+y)-Kf(x)- y \ind_{(0,1)}(y) (Kf)'(x)]$ then splits in three basic blocks that are:
\begin{align}&x[Kf(x+y)-Kf(x)- y \ind_{(0,1)}(y) (Kf)'(x)]\\
&= \sum_{\ell \geq 0} \expp{-\lambda x} x \left[ \expp{-\lambda y} \frac{(\lambda (x+y))^\ell}{\ell !} f(x+y, \ell)-
\frac{(\lambda x)^\ell}{\ell !} f(x, \ell)\right] \label{eq:CB-01}\\
&\quad - x y \ind_{(0,1)}(y) 
\left[
- \sum_{\ell \geq 0}  \lambda  \expp{-\lambda x} \frac{(\lambda x)^\ell}{\ell !} f(x, \ell) +
\sum_{\ell \geq 1} \lambda  \expp{-\lambda x} \frac{(\lambda x)^{\ell-1}}{(\ell-1)!} f(x, \ell) 
\right] \label{eq:CB-02}\\
&\quad - \sum_{\ell \geq 0}   \expp{-\lambda x} \frac{(\lambda x)^\ell}{\ell !}  xy \ind_{(0,1)}(y) 
f'(x, \ell). \label{eq:CB-03}
\end{align}
We expand \reff{eq:CB-01} using Newton binomial formula :
\begin{align}
& \sum_{\ell \geq 0} \expp{-\lambda x} x 
\left[ \expp{-\lambda y} \frac{\lambda^\ell \sum_{0 \leq  k \leq \ell} {\ell \choose k} x^{\ell-k} y^{k}}{\ell !} f(x+y, \ell)-
(\expp{-\lambda y}+ 1-\expp{-\lambda y}) \frac{(\lambda x)^\ell}{\ell !} f(x, \ell)\right] \nonumber \\
&= \sum_{\ell \geq 0} \expp{-\lambda x} \frac{(\lambda x)^{\ell}}{\ell!} 
\quad x \expp{-\lambda y}  [f(x+y, \ell)-f(x, \ell)] \label{eq:CB-1} \\
&+ \sum_{\ell \geq 1}\sum_{1 \leq  k \leq \ell}  \expp{-\lambda x}  \frac{(\lambda x)^{\ell+1-k}}{(\ell+1-k)!}  (\ell+1-k)y 
\expp{-\lambda y} \frac{(\lambda y)^{k-1}}{k!} 
f(x+y, \ell)   \label{eq:CB-11}  \\
&-  \sum_{\ell \geq 0} \expp{-\lambda x} \frac{(\lambda x)^{\ell}}{\ell!}  x (1-\expp{-\lambda y})  f(x, \ell) \label{eq:CB-comp}.
\end{align}

The second term \reff{eq:CB-11}  can be rewritten by a reindexation of the double sum as follows:
\begin{align}
&\sum_{\ell \geq 1}\sum_{1 \leq  k \leq \ell}  \expp{-\lambda x}   \frac{(\lambda x)^{\ell+1-k}}{(\ell+1-k)!} (\ell+1-k) y 
\expp{-\lambda y} \frac{(\lambda y)^{k-1}}{k!} 
f(x+y, \ell) \nonumber \\
=&
\sum_{\ell \geq 0,  k \geq 0} \expp{-\lambda x} \frac{(\lambda x)^{\ell}}{\ell!}  \, \ell 
y 
\expp{-\lambda y} \frac{(\lambda y)^{k}}{(k+1)!} 
f(x+y, \ell+k).  \label{eq:CB-2} 
\end{align}

Recollecting our findings so far, one has\\
\\
$x[Kf(x+y)-Kf(x)- y \ind_{(0,1)}(y) (Kf)'(x)]$
\begin{align}
=&\sum_{\ell \geq 0}   \expp{-\lambda x} \frac{(\lambda x)^\ell}{\ell!}  x \expp{-\lambda y}  [f(x+y,\ell)-f(x,\ell)]\label{53}\\
&+ \sum_{\ell \geq 0}  \expp{-\lambda x} \frac{(\lambda x)^\ell}{\ell !} \; \; \ell  \; \sum_{k \geq 0} y \expp{-\lambda y} \frac{(\lambda y)^k}{(k+1)!} f(x+y,\ell+k) \label{54}\\
&-\sum_{\ell \geq 0}  \expp{-\lambda x} \frac{(\lambda x)^\ell}{\ell !} x(1-\expp{-\lambda y})f(x,l)\label{55}\\
&-xy 1_{[0,1]}(y)\left(-\sum_{\ell \geq 0}\lambda \expp{-\lambda x}\frac{(\lambda x)^\ell}{\ell !} f(x,\ell)+\sum_{\ell \geq 1}\lambda \expp{-\lambda x}\frac{(\lambda x)^{\ell-1}}{(\ell-1) !} f(x,\ell)\right)\label{56}\\
&-\sum_{\ell \geq 0}  \expp{-\lambda x} \frac{(\lambda x)^\ell}{\ell !} xy1_{[0,1]}(y)f'(x,l)\label{57}.
\end{align}
Reordering the terms by gathering \eqref{53} and \eqref{57}, \eqref{54} with the second part of \eqref{56} and \eqref{55} with the first part,  we get:
\\
\\ 
$x[Kf(x+y)-Kf(x)- y \ind_{(0,1)}(y) (Kf)'(x)]$
\begin{align}
=&\sum_{\ell \geq 0}   \expp{-\lambda x} \frac{(\lambda x)^\ell}{\ell!}  x \expp{-\lambda y}  [f(x+y,\ell)-f(x,\ell)] \label{0}\\
&-\sum_{\ell \geq 0} \expp{-\lambda x}\frac{(\lambda x)^{\ell}}{\ell!} xy 1_{(0,1]}(y) f'(x,\ell) \label{4}\\
&+ \sum_{\ell \geq 0}  \expp{-\lambda x} \frac{(\lambda x)^\ell}{\ell !} \; \; \ell  \; \sum_{k \geq 0} y \expp{-\lambda y} \frac{(\lambda y)^k}{(k+1)!} f(x+y,\ell+k) \label{2}\\
&-\sum_{\ell \geq 1}\lambda \expp{-\lambda x}xy 1_{[0,1]}(y)\frac{(\lambda x)^{\ell-1}}{(\ell-1) !} f(x,\ell)\label{3} \\
&-\sum_{\ell \geq 0}  \expp{-\lambda x} \frac{(\lambda x)^\ell}{\ell!} x(1-\expp{-\lambda y}-\lambda y 1_{[0,1]}(y))f(x,\ell)\label{1}.
\end{align}
Equation \reff{4} can be rewritten as
\begin{align*}
\reff{4} &= -x\sum_{\ell \geq 0}\expp{-\lambda x}\frac{(\lambda x)^{\ell}}{\ell!}\left(\expp{-\lambda y}+1-\expp{-\lambda y}\right)y1_{(0,1]}(y)f'(x,\ell)\\
&=-x\sum_{\ell \geq 0}\expp{-\lambda x}\frac{(\lambda x)^{\ell}}{\ell!}\left(\expp{-\lambda y}y1_{(0,1]}(y)f'(x,\ell)+(1-\expp{-\lambda y})y1_{(0,1]}(y)f'(x,\ell)\right).
\end{align*}
Integrating  \reff{0}+\reff{4} with respect to $\nu$ yields 
$$x\sum_{\ell \geq 0}\expp{-\lambda x}\frac{(\lambda x)^{\ell}}{\ell!}\int_{(0,\infty}\expp{-\lambda y}\nu(\ddr y)\left(f(x+y,\ell)-f(x,\ell)-y1_{(0,1]}(y)f'(x,\ell)\right)+\psi'(\lambda)f'(x,\ell)=KG^{\psi_{\lambda}}f(x).$$
Notice that 

\[\reff{3}=-\sum_{\ell \geq 0}e^{-\lambda x}y\ind_{(0,1)}(y)\ell \frac{(\lambda x)^{\ell}}{\ell!}f(x,\ell).\]
Using the following identity
$$ y\expp{-\lambda y} \sum_{k \geq 0}  \frac{(\lambda y)^{k}}{(k+1)!}=
\expp{-\lambda y} \frac{1}{\lambda} \sum_{k \geq 1}  \frac{(\lambda y)^{k}}{k!} =\frac{1}{\lambda} (1- \expp{-\lambda y}).
$$
one can rewrite \reff{2}+\reff{3} as follows:

\begin{align}
\reff{2}+\reff{3}&=\sum_{\ell \geq 0}  \expp{-\lambda x} \frac{(\lambda x)^\ell}{\ell !} \; \; \ell  \left[ \sum_{k \geq 0} y \expp{-\lambda y} \frac{(\lambda y)^k}{(k+1)!} f(x+y,\ell+k)- y 1_{(0,1]}(y)f(x,\ell)\right]\\
&=\sum_{\ell \geq 0}  \expp{-\lambda x} \frac{(\lambda x)^\ell}{\ell !} \; \; \ell \sum_{k \geq 0} y \expp{-\lambda y} \frac{(\lambda y)^k}{(k+1)!}\left(f(x+y,\ell+k)-f(x,\ell)\right) \label{gen2}
\\
&\qquad + \sum_{\ell \geq 0}  \expp{-\lambda x} \frac{(\lambda x)^\ell}{\ell !} \; \; \ell \left(\frac{1-\expp{-\lambda y}}{\lambda}-y1_{(0,1]}(y)\right)f(x,\ell)
\\
&=\reff{gen2}+x\sum_{\ell' \geq 0}e^{-\lambda x}\frac{(\lambda x)^{\ell'}}{\ell'!}\left[1-\expp{-\lambda y}-\lambda y 1_{(0,1]}(y)\right]f(x, \ell'+1).
\end{align}
We sum \reff{2}+\reff{3}+\reff{1} and obtain:
\begin{equation}\label{finalidentity}
\reff{2}+\reff{3}+\reff{1}= \reff{gen2}-x\left(\expp{-\lambda y}-1+\lambda y 1_{(0,1]}(y)\right)\sum_{\ell\geq 0}e^{-\lambda x}\frac{(\lambda x)^{\ell}}{\ell !}[f(x, \ell+1)-f(x,\ell)].\end{equation}

By integrating with respect to $\nu$, the term \reff{gen2}, one has 
$$\sum_{\ell \geq 0}  \expp{-\lambda x} \frac{(\lambda x)^\ell}{\ell !} \; \; \ell  \; \sum_{k \geq 0} \int_{(0,\infty)}   \nu(\ddr y) y \expp{-\lambda y} \frac{(\lambda y)^k}{(k+1)!} [f(x+y,\ell+k)-f(x,\ell)],$$
and we recognize the cross jump term $\mathcal{J}^{NL}f(x,\ell)$, see  \eqref{JNL}.\\

By integrating with respect to $\nu$, the second part of the right-hand side in \eqref{finalidentity}, we get 
\begin{equation}\label{extrabirth}
-x\psi(\lambda)\sum_{\ell\geq 0}e^{-\lambda x}\frac{(\lambda x)^{\ell}}{\ell!}[f(x,\ell+1)-f(x,\ell)]=K\mathcal{R}f(x).
\end{equation}
Finally, we have established
\[\mathcal{G}Kf(x)=K\big(\mathrm{c}+\mathcal{G}^{\psi_\lambda}+\mathcal{J}^{NL}+\mathcal{R}\big)f(x)=K\mathcal{H}f(x).\]
\begin{proof}[Proof of Theorem \ref{thm1:main}] Only remains to gather the diffusive part and the jump part, this is a direct consequence of linearity and the fact that $\mathcal{J}^{L}+\mathcal{J}^{NL}=\mathcal{J}$. One has indeed
\[\mathcal{H}=\mathrm{c}+\mathcal{G}^{\psi_\lambda}+\mathcal{J}^{L}+\mathcal{J}^{NL}+\mathcal{R},\]
thus 
\[K\mathcal{H}f=K\left(\mathrm{c}+\mathcal{G}^{\psi_\lambda}+\mathcal{J}^{L}+\mathcal{J}^{NL}+\mathcal{R}\right)f=\mathcal{G}Kf.\]
\end{proof}

\section{Proofs of Corollary \ref{cor:identifyjointbranchingmech}, Theorem \ref{thm:main} and Proposition \ref{prop:explosion}}
Let $\psi$ be a branching mechanism. Let $\lambda>0$ and $(X^{\lambda},L^{\lambda})$ be the two-type branching process with generator $\mathcal{H}$, see \eqref{Hi} and \eqref{Hii}. We start by studying the joint branching mechanism $\mathbf{\Psi}=(\Psi_c,\Psi_d)$ and then establish the skeleton decomposition, that is to say Theorem \ref{thm:main}. Proposition \ref{prop:explosion} about the explosion will be a consequence and is proved at the end of the section. We stress to the reader that the proof of Theorem \ref{thm:main} does not appeal to Corollary \ref{cor:identifyjointbranchingmech} but only to Theorem \ref{thm1:main} and Theorem~\ref{thm:branchingbitype}.  
\subsection{Proof of Corollary \ref{cor:identifyjointbranchingmech}: study of the joint branching mechanism}\label{sec:proofprop1}
Notice that $\mathbf{\Psi}$ depends on $\lambda$. For any $q\in (0,\infty)$, define on $[0,\infty)$, the map $e_q(x):=e^{-qx}$. The identities \eqref{Psic-(i)} and \eqref{Psic-(ii)}: \[\Psi_c(q,r)=\psi_\lambda(q)-\kappa \text{ if } \lambda\geq \rho \text{ and }  \Psi_c(q,r)=\psi_\lambda(q)-\psi(\lambda)(r-1)-\kappa \text{ if } \lambda<\rho,\]
follows readily from the definition of the generator $\mathcal{H}$, see \eqref{Hi}-\eqref{Hii} and the fact that $\mathcal{G}^{\psi_\lambda}e_q(x)=x\psi_{\lambda}(q)e_q(x)=-\psi_{\lambda}(q)\frac{\ddr }{\ddr q}e_q(x)$. The extra birth term when $\lambda<\rho$ comes from the last term in \eqref{Hii}.  The expression for $\Psi_d$ given by \eqref{Psid-(i)}-\eqref{Psid-(ii)} according whether $\lambda\geq \rho$ or $\lambda<\rho$ follows by definition of the measure $\rho(\ddr y,\ddr k)$. The extra death term when $\lambda\geq \rho$ comes from the last term in \eqref{Hi}.
The identity 
\begin{equation}\label{eq:lemrelationship-psiPsi}
\psi\big(q+\lambda(1-r)\big)=\Psi_c(q,r)+\lambda\Psi_d(q,r).
\end{equation}
is a consequence of the algebraic intertwining relationship \eqref{eq:intertwiningHKG} applied to the function $f_{q,r}(x,\ell)=e^{-qx}r^{\ell}$. Indeed, one has $Kf_{q,r}(x)=e^{-(q+\lambda(1-r))x}=e_{q+\lambda(1-r)}(x)$, hence on the one hand
\[\mathcal{G}Kf_{q,r}(x)=\mathcal{G}e_{q+\lambda(1-r)}(x)=x\psi\big(q+\lambda(1-r)\big)e_{q+\lambda(1-r)}(x).\]
On the other hand, recall 
\[\mathcal{H}f_{q,r}(x,\ell)=x\Psi_c(q,r)e^{-qx}r^{\ell}+e^{-qx}\ell r^{\ell -1}\Psi_d(q,r),\]
see \eqref{eq:Lfqr}, and thus
\begin{align*}
K\mathcal{H}f_{q,r}(x)&=\sum_{\ell\geq 0}e^{-\lambda x}\frac{(\lambda x)^{\ell}}{\ell !}\mathcal{H}f_{q,r}(x,\ell)\\
&=\Psi_c(q,r)xe^{-(\lambda(1-r)+q)x}+e^{-qx}\Psi_d(q,r)e^{-\lambda x}\frac{\ddr}{\ddr r}e^{\lambda xr}\\
&=xe_{q+\lambda(1-r)}(x)\big(\Psi_c(q,r)+\lambda \Psi_d(q,r)\big).
\end{align*}
Now the equality $\mathcal{G}Kf_{q,r}=K\mathcal{H}f_{q,r}$ entails \eqref{eq:lemrelationship-psiPsi}. The identities \eqref{Psid-(i)2} and \eqref{Psid-(ii)2} follow readily by this relationship and \eqref{Psic-(i)} and \eqref{Psic-(ii)}. \qed

\subsection{Proof of Theorem \ref{thm:main}: intertwining of semigroups} \label{sec:proofmainthm}
Denote by $(Q_t)$ the semigroup of a CSBP$(\psi)$.
Call $(P_t)$ the semigroup of the process $(X^{\lambda},L^{\lambda})$.
\begin{lem}[Intertwining of semigroups]\label{lem:intertwinedsemigroup}\
\vspace*{1mm}
For all $f\in C_0(\mathbb{R}_+\times \mathbb{Z}_+)$,
\begin{equation}\label{intertwinedsemigroupfqr}KP_tf(x)=Q_tKf(x), \qquad \forall t,x\geq 0.
\end{equation} 
\end{lem}

\begin{proof}
Recall Proposition \eqref{Fellerprop} and the definition of the cores $D$ and $D_c$ in \eqref{eq:coreD} and \eqref{Dexpoc} respectively. Let $f\in D$. Note that $Kf\in D_c$ and the map $(t,x)\mapsto Q_t(Kf)(x)$ is the unique solution to the \textit{backward} Kolmogorov equation:
\[\frac{\ddr}{\ddr t }Q_tKf=\mathcal{G}Q_tKf,\ Q_0Kf=Kf,\]
see Subsection \ref{sec:csbps}. Similarly,  the semigroup $(P_t)$ of $\mathbf{X}=(X^{\lambda},L^{\lambda})$  satisfies the \textit{backward} Kolmogorov equation:
\[\frac{\ddr}{\ddr t }P_tf=\mathscr{L}P_tf,\ P_0f=f.\]
Proposition \ref{Fellerprop} ensures that $P_tf\in D\subset \mathcal{D}$. Then the algebraic relationship given by Theorem \ref{thm1:main} ensures that for any $t\geq 0$, $K\mathscr{L}P_tf=\mathcal{G}KP_tf$. The derivative being in a uniform sense, one can interchange it with the Poisson kernel $K$ and we see that
\[\frac{\ddr}{\ddr t }KP_tf=K\frac{\ddr}{\ddr t }P_tf=K\mathscr{L}P_tf=\mathcal{G}KP_tf.\]
Last, since $KP_0f=Kf$, by uniqueness of the solution of the backward equation solved by $Q_t$, we have for all $t\geq 0$,
\[KP_tf=Q_tKf.\]
The fact that this holds for any $f\in C_0(\mathbb{R}_+\times \mathbb{Z}_+)$ follows by density. 
\end{proof}

The intertwining  of semigroups immediately extends to the kernel $\Lambda$ defined in \eqref{eq:defLambda}.

\begin{lem}[Intertwining of semigroups]\label{lem:intertwinedsemigroup-Lambda}\
\vspace*{1mm}
For all $f\in C_0(\mathbb{R}_+\times \mathbb{Z}_+)$,
\begin{equation}\label{intertwinedsemigroupfqr-Lambda}\Lambda P_tf(x)=Q_t \Lambda f(x), \qquad \forall t \in  \R_{+} ,x\in \R_{+} \cup \{\infty\}.
\end{equation} 
\end{lem}

\begin{proof}
From the definition of $\Lambda$ in term of $K$ in \eqref{eq:defLambda}, we get:
$$
\Lambda P_t f(x) = 
\sum_{\ell \in \mathbb Z_{+}} \Lambda(x,(x,\ell)) P_t f(x,\ell)  +  \Lambda(x,\Delta) P_t f(\Delta)
=\ind_{[0,\infty)}(x)K P_t f(x)  +  \ind_{\{\infty\}}(x) f(\Delta)
$$
while 
$$Q_t \Lambda f(x) = Q_t\big(
\sum_{\ell \in \mathbb Z_{+}} \Lambda(x,(x,\ell)) f(x,\ell) + 
 \Lambda(x,\Delta) f(\Delta)\big)= \ind_{[0,\infty)}(x) Q_t Kf(x) +  \ind_{\{\infty\}}(x) f(\Delta)
$$
and \eqref{intertwinedsemigroupfqr-Lambda} now follows from \eqref{intertwinedsemigroupfqr}.
\end{proof}

\begin{proof}[Proof of Theorem \ref{thm:main}] The first statement is given by Lemma \ref{lem:intertwinedsemigroup}. The two other assertions follow by applying Pitman-Rogers theorem, see Theorem \ref{thmRG}.  Precisely, we set 
$S=(\mathbb{R}_+\times \mathbb{Z}_+)\cup \{\Delta\}$
and
$S'=\mathbb{R}_+\cup\{\infty\}= \bar{\mathbb{R}}_+$, then the map $\phi: S \to S'$  and the kernel $\Lambda$ from $S'$ to $S$ are defined by:
$$\phi(x,\ell):=x , \qquad  \Lambda\big(x,(x,\ell)\big):=K(x,\ell)
\ind_{[0,\infty)}(x)+\delta_{\Delta}\big(x,\ell)\ind_{\{\infty\}}(x).$$
Last $P_t$ and $Q_t$ are the Markov semigroups  defined  on $S$ and $S'$ before Lemma \ref{lem:intertwinedsemigroup}, and $\Phi$ be defined as the operator that acts on bounded measurable functions on $S'$ by right composition of $\phi$, namely
$\Phi(f)= f \circ \phi$.
With these notations, $\Lambda \Phi$ is indeed the identity kernel on $S'$, since $\Lambda(x, \phi^{-1}(x))=1$.
Second, Lemma \ref{lem:intertwinedsemigroup} together with Remark \ref{rem:continuousbounded} ensure that 
$\Lambda P_t = Q_t \Lambda$ is satisfied for each $t \geq 0$, which in turn, composing on the right by $\Phi$ ensures that $Q_t$ is indeed defined from $P_t$ by 
$Q_t =K P_t \Phi$ for each $t \geq 0$. The assumptions of Pitman-Rogers criterion, Theorem \ref{thmRG}, are thus met. The first statement of Theorem \ref{thm:main} follows from  Theorem \ref{thmRG}-(1). Its second statement follows by Theorem \ref{thmRG}-(2), see
\eqref{eq:cond-law-strong}, with $A=\{\ell\}\times \bar{\mathbb{R}}_+$.
\end{proof}

We now study the phenomenon of explosion. 

\subsection{Proof of Proposition \ref{prop:explosion}: explosion along skeletons}\label{sec:proofprop2} Recall that our aim is to establish that when the CSBP $X^{\lambda}$ explodes continuously (i.e. not by a single jump to $\infty$), it does simultaneously as any of its skeleton.
We work under the assumption $\psi(0)=-\kappa=0$, so that no killing is allowed.  Recall that the cemetery point of the process $(X^{\lambda},L^{\lambda})$ is $\Delta=\{(x,\ell): x+\ell=\infty\}$. Since $f_{q,r}(x,\ell)=e^{-qx}r^{\ell}\underset{(q,r)\rightarrow (0,1)}{\longrightarrow} \ind_{\mathbb{R}_+\times \mathbb{Z}_+}(x,\ell)$ and $K\ind_{\mathbb{R}_+\times \mathbb{Z}_+}=\ind_{\mathbb{R}_+}$, we see from \eqref{intertwinedsemigroupfqr} that 
\begin{equation}\label{eq:KPt1}
K(P_t\ind_{\mathbb{R}_+\times \mathbb{Z}_+})(x)=Q_t(\ind_{\mathbb{R}_+})(x).  
\end{equation}
At a probabilistic level, set $\mathbb{P}:=\mathbb{P}_{x,\mathrm{Poi}(\lambda x)}$, define $\zeta:=\inf\{t>0: (X^{\lambda}_{t-},L^{\lambda}_{t-})=\Delta\}$ and $\zeta_c:=\inf\{t>0:X_{t-}^{\lambda}=\infty\}$. By definition, one has $\zeta\leq \zeta_c$ a.s. and \eqref{eq:KPt1} entails the equality
\[\mathbb{P}(\zeta>t)=\mathbb{P}(\zeta_c>t), \ \forall t>0,\]
which in turn ensures that $\mathbb{P}(\zeta=\zeta_c)=1$. We now show that
$\zeta=\zeta_d:=\inf\{t>0:L_{t-}^{\lambda}=\infty\}$.
Denote by $\zeta_n^{+}:=\inf\{t>0:X_t^{\lambda}>n\}$, the first passage time above $n$. Notice that for all $n\geq 1$, $\zeta_n^{+}<\zeta_c=\underset{n\rightarrow \infty}{\lim} \! \uparrow \zeta_n^{+}$.
By Theorem \ref{thm:main}, conditionally on $\zeta_n^{+}<\infty$ and $X_{\zeta_n^{+}}^{\lambda}$, $L_{\zeta_n^{+}}^{\lambda}$ has law $\mathrm{Poi}(\lambda X_{\zeta_n^{+}}^{\lambda})$. Thus, for any $\ell\geq 0$,
\[\mathbb{P}\big(L^{\lambda}_{\zeta_n^{+}}\leq \ell|X^{\lambda}_{\zeta_n^{+}}\big)=e^{-\lambda X^{\lambda}_{\zeta_n^{+}}}\sum_{k=0}^{\ell}\frac{(\lambda X^{\lambda}_{\zeta_n^{+}})^{k}}{k!}\leq 2^{\ell}e^{-\frac{\lambda}{2} X^{\lambda}_{\zeta_n^{+}}}\leq 2^{\ell}e^{-n\lambda/2} \text{ a.s. on }\{\zeta_{c}<\infty\}.\]
By taking expectation and then letting $n$ go to $\infty$, we see that for all $\ell\in \mathbb{Z}_+$
$$\mathbb{P}(L_{\zeta_c-}^{\lambda}\leq \ell,\zeta_c<\infty)=0.$$
We conclude that $\mathbb{P}$-almost surely
$\zeta_d\leq \zeta_c$. Therefore, $\mathbb{P}(\zeta=\zeta_d)=1$ and the proof is achieved. \qed

\section{Proofs of Proposition \ref{prop:mechanismskeleton} and Theorem \ref{thm:convergence}}\label{sec:proofsonskeletons}
We study in this section the skeletons when $\lambda\geq \rho$. Recall that in this case, they are autonomous. We start by identifying their branching mechanisms. We establish next that once rescaled by $\lambda$ they converge in the Skorohod sense towards the CSBP$(\psi)$.
\subsection{Proof of Proposition \ref{prop:mechanismskeleton}: skeletons offspring distributions}\label{sec:proofprop:mechanismsskeleton}
First, by Corollary \ref{cor:identifyjointbranchingmech}, we see, by letting $q$ go to $0$, in \eqref{Psid-(i)2}, that when $\lambda\geq \rho$, $\Psi_d(0,r)=\frac{\psi(\lambda(1-r))+\kappa}{\lambda}$. We now look for the offspring distribution. The form of $\Psi_d$ in \eqref{Psid-(i)} entails
\begin{align*}\Psi_d(0,r)&=\sum_{k\geq 1}(r^{k+1}-r)\int_{0}^{\infty}ye^{-\lambda y}\frac{(\lambda y)^k}{(k+1)!}\nu(\ddr y)+\frac{\sigma^2}{2}\lambda(r^{2}-r)+\frac{\psi(\lambda)}{\lambda}(1-r).
\end{align*}
It remains to compute the total rate of branching. Namely
\[B:=\sum_{k\geq 1}\int_{0}^{\infty}ye^{-\lambda y}\frac{(\lambda y)^k}{(k+1)!}\nu(\ddr y)+\frac{\sigma^2}{2}\lambda+\frac{\psi(\lambda)}{\lambda}=:B_1+\frac{\sigma^2}{2}\lambda+\frac{\psi(\lambda)}{\lambda}.\]
Plainly
\begin{align}
B_1&=\frac{1}{\lambda}\int_{0}^{\infty} e^{-\lambda y}\sum_{k=1}^{\infty}\frac{(\lambda y)^{k+1}}{(k+1)!}\nu(\ddr y)=\frac{1}{\lambda} \int_{0}^{\infty}\left(1-e^{-\lambda y}-\lambda y e^{-\lambda y}\right) \nu(\ddr y) \nonumber \\
&=\frac{1}{\lambda}\left(\int_0^1(1-e^{-\lambda y}-\lambda y)\nu(\ddr y)+\int_{0}^{1}\lambda y(1-e^{-\lambda y})\nu(\ddr y)\right. \nonumber \\
&\qquad \qquad \qquad \left.+\int_{1}^{\infty}(1-e^{-\lambda y})\nu(\ddr y)-\int_1^{\infty}\lambda ye^{-\lambda y}\nu(\ddr y)\right). \label{decompoB1}
\end{align}
One has
\begin{equation}\label{psilsurl}\frac{\psi(\lambda)}{\lambda}=\frac{\sigma^2}{2}\lambda-\gamma+\frac{1}{\lambda}\left(\int_{0}^{1}\big(e^{-\lambda y}-1+\lambda y\big)\nu(\ddr y)+\int_{1}^{\infty}\big(e^{-\lambda y}-1\big)\nu(\ddr y)\right).
\end{equation}
Thus, adding \eqref{decompoB1} and \eqref{psilsurl} and canceling the common terms, we obtain
\begin{align*}
B=B_1+\frac{\sigma^2}{2}\lambda+\frac{\psi(\lambda)}{\lambda}&=\sigma^2 \lambda-\gamma+\int_{0}^{1}y(1-e^{-\lambda y})\nu(\ddr y)-\int_{1}^{\infty}ye^{-\lambda y}\nu(\ddr y)=\psi'(\lambda).
\end{align*}
\qed

We now establish Theorem \ref{thm:convergence}. 
\subsection{Proof of Theorem \ref{thm:convergence}: scaling limits of skeletons}
Denote by $[0,\infty]$, the extended half-line endowed with the metric $d(x,y):=|e^{-x}-e^{-y}|$, with convention $e^{-\infty}=0$. Call $\mathbb{D}$ the Skorohod space of $[0,\infty]$-valued càdlàg paths.   Denote by $\mathbf{P}_x$ the probability law on $\mathbb{D}$ associated to the skeleton process $(L_t^{\lambda},t\geq 0)$ issued from a Poisson number of individuals:  $L_0^{\lambda}\overset{\text{law}}{=}\mathrm{Poi}(\lambda x)$.
\begin{lem}\label{lem:convonexpo} For all $q\in (0,\infty)$, 
\[\mathbf{E}_x\left[e^{-qL_t^{\lambda}/\lambda}\right]\underset{\lambda \rightarrow \infty}{\longrightarrow} e^{-xu_t(q)}, \text{ uniformly in } x\in [0,\infty).\]
\end{lem}
\begin{proof}
Recall $\varphi_\lambda$ in \eqref{varphi}. Define
$f^{\lambda}_t(q):=\mathbb{E}\left(e^{-qL_t^{\lambda}/\lambda}|L_0^{\lambda}=1\right)$. One has 
\begin{align}
\mathbf{E}_x\left[e^{-qL_t^{\lambda}/\lambda}\right]&= \sum_{k=0}^{\infty}\frac{(\lambda x)^k}{k!}e^{-\lambda x}\mathbb{E}\left[e^{-qL_t^{\lambda}/\lambda}|L_0^{\lambda}=k\right]\nonumber\\
&= \sum_{k=0}^{\infty}\frac{\big(\lambda xf^{\lambda}_t(e^{-q/\lambda})\big)^k}{k!}e^{-\lambda x}\nonumber\\
&=e^{-\lambda x\big(1-f^{\lambda}_t(e^{-q/\lambda})\big)}.\label{eq:cumulantskeleton}
\end{align}
The map $t\mapsto v_t^{\lambda}(q):=\lambda\big(1-f^{\lambda}_t(e^{-q/\lambda})\big)$ satisfies the o.d.e
\begin{align*}
&\frac{\ddr }{\ddr t}v_t^{\lambda}(q)=-\lambda \varphi_{\lambda}\big(f_t^{\lambda}(e^{-q/\lambda})\big)=-\psi\big(v_t^{\lambda}(q)\big)\\
&v^{\lambda}_0(q)=\lambda\big(1-e^{-q/\lambda}\big).
\end{align*}
Therefore, by uniqueness,  \[v_t^\lambda(q)=u_t\big(\lambda(1-e^{-q/\lambda})\big), \text{ for all } t\geq 0.\]
By continuity of $\eta\mapsto u_t(\eta)$, we see that $v_t^{\lambda}(q)\underset{\lambda \rightarrow \infty}{\longrightarrow} u_t(q)$. Thus, for any $x\geq 0$,
\[\mathbf{E}_x\left[e^{-qL_t^{\lambda}/\lambda}\right]\underset{\lambda \rightarrow \infty}{\longrightarrow} e^{-xu_t(q)}.\]
It remains to establish the uniform convergence, and, as we shall see, its order can be determined as well:
\begin{align}
\sup_{x\geq 0}\left\lvert\mathbf{E}_x\left[e^{-qL_t^{\lambda}/\lambda}\right]-e^{-xu_t(q)}\right\lvert&=\sup_{x\geq 0}\left\lvert e^{-xu_t\left(\lambda(1-e^{-q/\lambda})\right)}-e^{-xu_t(q)}\right\lvert\nonumber\\
&=\sup_{x\geq 0}e^{-xu_t(q)}\left\lvert e^{-x\big(u_t\left(\lambda(1-e^{-q/\lambda})\right)-u_t(q)\big)}-1\right\lvert\nonumber\\
&\leq \sup_{x\geq 0}xe^{-xu_t(q)}\left\lvert u_t\left(\lambda(1-e^{-q/\lambda})\right)-u_t(q)\right\lvert\nonumber\\
&= \frac{1}{u_t(q)}\left\lvert u_t\left(\lambda(1-e^{-q/\lambda})\right)-u_t(q)\right\lvert \nonumber\\
&\leq \frac{1}{u_t(q)}\frac{\partial}{\partial q}u_t(q)\left(q-\lambda\left(1-e^{-q/\lambda}\right)\right)=o(1/\lambda),\nonumber
\label{uniformconv}
\end{align}
where $o(1/\lambda)$ is a positive function $g$ such that $\underset{\lambda \rightarrow \infty}{\lim} \lambda g(\lambda)=0$.
\end{proof}
Denote by $\mathbf{Q}_t^{1/\lambda}$ the semigroup of $(L_t^{\lambda}/\lambda,t\geq 0)$ under the probability laws $(\mathbf{P}_x,x\in [0,\infty))$. Recall $Q_t$ the semigroup of the CSBP$(\psi)$. We denote below the supremum norm by $\|\cdot\|_{\infty}$.
\begin{lem}\label{lem:unifconvsemigroup} For any $f\in C_0(\bar{\mathbb{R}}_+)$, 
\[\|\mathbf{Q}_t^{1/\lambda}f-Q_tf\|_{\infty}\underset{\lambda \rightarrow \infty}{\longrightarrow} 0.\]
\end{lem}
\begin{proof}
Let $f\in C_0(\bar{\mathbb{R}}_+)$ and $\epsilon>0$. By density, there exists $g$ in $D_c$, the linear span of $\{e_q(\cdot),q>0\}$, such that $\|f-g\|_\infty\leq \epsilon/4$. Recall that $Q_te_q(x)=e^{-xu_t(q)}$. By Lemma \ref{lem:convonexpo}, one can choose $\lambda$ large enough so that
$\|\mathbf{Q}_t^{1/\lambda}g-Q_tg\|_{\infty}\leq \epsilon/2$ and one has
\begin{align*}
\|\mathbf{Q}_t^{1/\lambda}f-Q_tf\|_{\infty}&=\|\mathbf{Q}_t^{1/\lambda}f-\mathbf{Q}_t^{1/\lambda}g+\mathbf{Q}_t^{1/\lambda}g+Q_tg-Q_tf-Q_tg\|_{\infty}\\
&\leq 2\|f-g\|_\infty+\|\mathbf{Q}_t^{1/\lambda}g-Q_tg\|_{\infty}\leq \epsilon.
\end{align*}
\end{proof}
\begin{proof}[Proof of Theorem \ref{thm:convergence}] This is obtained by combining Lemma \ref{lem:unifconvsemigroup} and \cite[Theorem 2.5 p.167]{EK}, which states that uniform convergence of a family of Feller semigroups, together with the convergence of the initial laws, entails the convergence in Skorohod's sense of the associated processes.\end{proof}
\section*{Appendix}

Let $(S, \mathcal S)$, $(S', \mathcal S')$ be two measurable spaces, with their respective set of bounded and measurable functions denoted by $b\mathcal S$ and $b \mathcal S'$. Let $(P_t)_{t \geq 0}$ be a Markov semigroup on $(S, \mathcal S)$, and $\phi : S \to S'$ be measurable. The action of $\phi$ by right-composition defines an operator $\Phi : b\mathcal S' \to b\mathcal S, f \mapsto \Phi(f):=f\circ \phi$. Denote by $\mathbf{X}$ the Markov process with semigroup $(P_t)$ and denote its image by the mapping $\phi$,  $\phi(\mathbf{X}):=(\phi(\mathbf{X}_t),t\geq 0)$. Call $\mathcal F^{\phi(\mathbf{X}),0}_t=\sigma\big(\phi(\mathbf{X}_s), 0 \leq s \leq t\big)$ the natural filtration generated by $\phi(\mathbf{X})$. Let $\mathcal F^{\phi(\mathbf{X})}_t$ be its usual augmentation. 

\begin{theostar}[Pitman-Rogers criterion, Theorem 2 in \cite{RP81}]
\label{thmRG}
Suppose there is a Markov kernel $\Lambda$ from $(S', \mathcal{S}')$ to $(S,\mathcal{S})$ such that
\begin{itemize}
\item $\Lambda \Phi =I$, the identity kernel on $S'$.
\item  for each $t \geq 0$, the Markov kernel $Q_t :=\Lambda P_t \Phi$ from $S'$ to $S'$
satisfies the identity 
\begin{equation}
\label{inter-append}
\Lambda P_t = Q_t \Lambda.
\end{equation}

\end{itemize} 
\begin{enumerate}
    \item If $\mathbf{X}$ has initial distribution $\Lambda(y, \cdot)$, with $y \in S'$, then for each $t \geq 0, A \in \mathcal{S}$, 
\begin{equation}
\label{eq:cond-law}
    \P(\mathbf{X}_t \in A | \mathcal F^{\phi(\mathbf{X}),0}_t) = \Lambda\big(\phi(\mathbf{X}_t), A\big) \qquad \text{a.s.,}
\end{equation}
holds, and $\phi(\mathbf{X})$ is Markov with starting state $y$ and transition semigroup $(Q_t)$. Additionally, \eqref{eq:cond-law} again holds with $\mathcal F^{\phi(\mathbf{X}),0}_t$ replaced by $\mathcal F^{\phi(\mathbf{X})}_t$.
\item If $\mathbf{X}$ has a.s. càdlàg sample paths, $\phi(\mathbf{X})$ is Feller and $x \mapsto \Lambda(x,\cdot)$ is weakly continuous, then for each $\mathcal F^{\phi(\mathbf{X})}_t$ stopping time $T$, it holds:
\begin{equation}
\label{eq:cond-law-strong}
    \P(\mathbf{X}_T \in A | \mathcal F^{\phi(\mathbf{X})}_T) = \Lambda\big(\phi(\mathbf{X}_T), A\big) \qquad \text{a.s. on the event } \{T<\infty\}.
\end{equation}

\end{enumerate}
\end{theostar}

Only the second statement in point (1) and point (2) are formally new. We give a complete proof of the theorem for the benefit of the reader. Regarding the second assumption, we stress that \textit{any} Markov kernel $(Q_t)$ satisfying the intertwining relationship \eqref{inter-append} necessarily satisfies, under the assumption $\Lambda \Phi =I$, that $Q_t = \Lambda P_t \Phi$.

\begin{proof}
We first establish (1). Let $y \in S'$. Given $f \in b\mathcal S'$ and $g \in b\mathcal S$, and using that $f(\phi(x)) = f(y)$, $\Lambda(y, .)$-a.s.,  we get (6) as follows 
$$\Lambda (\Phi f) g  (y) = \int_{S} \Lambda(y,\ddr x)  f(\phi(x)) g(x) =   \int_{S} \Lambda(y,\ddr x)  f(y) g(x)  = (f \Lambda g) (y).$$
Now, for $y \in S'$, the quantity $\Lambda  P_t (\Phi f) g$ evaluated at $y$ has to be interpreted as 
$\Lambda  P_t (\Phi f) g  (y) = \int_{S} \Lambda(y,\ddr x)  \int_S P_t(x,\ddr x')  ( f(\phi(x')) g(x') )$  
and, using that $\Lambda  P_t= Q_t \Lambda$, we arrive at the following equality, which corresponds to Eq (7) p. 574 in \cite{RP81}:
\begin{align*}
(\Lambda  P_t (\Phi f) g)  (y) 
& = (\Lambda  P_t)  ( f(\phi)   g )(y) \\
& = (Q_t \Lambda)  ( f(\phi) g )(y) \\
& = Q_t (\Lambda ( f(\phi)  g ))(y) \\
& = Q_t (\Lambda ( (\Phi f)  g ))(y) \\
& = Q_t (f \Lambda g) (y).
\end{align*}

As for the induction step, Eq (8), setting $h(x)= P_{t_{n}}( f_{n} \circ \phi  \cdot g)(x)$, and assuming the property holds at step $n-1$, we get:
\begin{align*}
\Lambda  P_{t_1} (\Phi f_1)  P_{t_2} (\Phi f_2)  \ldots P_{t_{n}} (\Phi f_{n})  g  
& =\Lambda  P_{t_1} (\Phi f_1)   \ldots   P_{t_{n-1}} (\Phi f_{n-1}) h  \\
& = Q_{t_1} f_1  \ldots Q_{t_{n-1}}  f_{n-1}  \Lambda h \\
& = Q_{t_1} f_1  \ldots Q_{t_{n-1}}  f_{n-1}  \Lambda P_{t_{n}} ( f_{n} \circ \phi   \cdot g )\\
& = Q_{t_1} f_1  \ldots Q_{t_{n-1}}  f_{n-1}  \Lambda P_{t_{n}}(\Phi(f_{n})  g)  \\
& = Q_{t_1} f_1 \ldots Q_{t_{n-1}}  f_{n-1}  Q_{t_{n}} (f_{n}) \Lambda   g
\end{align*}

Now, writing $\mathbb{E}_{\Lambda(y,.)}$ for $\int_x \Lambda(y,\ddr x) \mathbb{E}_{x}$, we get:

\begin{align*}
\mathbb{E}_{\Lambda(y,.)} [ f_1 (\phi(\mathbf{X}_{t_1}))  f_2 (\phi(\mathbf{X}_{t_1+t_2}))    f_{n+1}(\phi(\mathbf{X}_{t_1+\ldots+t_n})) g(X_{t_1+\ldots+t_n})]  
&=  \big( \Lambda  P_{t_1} (\Phi f_1)  P_{t_2} (\Phi f_2)  P_{t_{n}} (\Phi f_{n})  g \big)  (y) \\
& =  \big( Q_{t_1} f_1  \ldots Q_{t_{n-1}}  f_{n-1}  Q_{t_{n}} f_{n} \Lambda   g \big) (y)
\end{align*}
implies, with $g=1$, that $\phi \circ X$ starting from initial measure $\Lambda(y,.)$ is Markov with semi-group $Q_{t}$, 
while the identity:
\begin{align*}
& \mathbb{E}_{\Lambda(y,.)} \Big[ f_1 (\phi(\mathbf{X}_{t_1}))  f_2 (\phi(\mathbf{X}_{t_1+t_2}))   \ldots f_{n}(\phi(\mathbf{X}_{t_1+\ldots+t_n})) g(\mathbf{X}_{t_1+\ldots+t_n})\Big] \\
& = \mathbb{E}_{\Lambda(y,.)}  \Big[f_1 (\phi(\mathbf{X}_{t_1}))  f_2 (\phi(\mathbf{X}_{t_1+t_2}))  \ldots  f_{n}(\phi(\mathbf{X}_{t_1+\ldots+t_n}))  \int \Lambda(\mathbf{X}_{t_1+\ldots+t_n}, \ddr x)   g(x)\Big]
\end{align*}
and this in turn gives by a monotone class theorem that for each bounded or positive random variable $Z$, measurable with respect to $\mathcal F^{\phi(\mathbf{X}),0}_t$, 
$$\mathbb{E}_{\Lambda(y,.)} \Big[ Z \, g(\mathbf{X}_{t})\Big] \\
= \mathbb{E}_{\Lambda(y,.)}  \Big[ Z  \,  \int \Lambda(\mathbf{X}_{t}, \ddr x)   g(x)\Big].
$$
This entails that the conditional distribution of $\mathbf{X}_{t}$ given $\mathcal F^{\phi(\mathbf{X}),0}_t$ is $\Lambda(\mathbf{X}_t, \ddr x)$. Let now $Z$ be a positive  $\mathcal F^{\phi(\mathbf{X})}_t$-measurable random variable. By definition of the augmented filtration, there are two $\mathcal F^{\phi(X),0}_t$-measurable random variables $Z',Z''$ such that $Z'\leq Z\leq Z''$ and $\mathbb{P}_{\Lambda(y,\cdot)}(Z''-Z'>0)=0$. Hence, 
$\mathbb{E}_{\Lambda(y,\cdot)}[\ind_{\{Z''-Z'>0\}}g(\mathbf{X}_t)]=0$. With no loss of generality, assuming $g$ non-negative, one has
\[\mathbb{E}_{\Lambda(y,\cdot)}[Z'g(\mathbf{X}_t)]\leq \mathbb{E}_{\Lambda(y,\cdot)}[Zg(\mathbf{X}_t)]\leq \mathbb{E}_{\Lambda(y,\cdot)}[Z''g(\mathbf{X}_t)].\]
The lower and upper bounds are  $\mathbb{P}_{\Lambda(y,\cdot)}$-a.s. equal to $\int \Lambda(X_{t}, \ddr x)   g(x)$, which ends the proof of (1).

For statement (2). We start by considering a stopping time $T$ taking values in a countable set $D$. One has plainly,
\begin{align*}\ind_{\{T<\infty\}}\mathbb{E}_{\Lambda(y,\cdot)}[g(\mathbf{X}_T)|\mathcal{F}^{\Phi(\mathbf{X})}_T]&=\ind_{\{T<\infty\}} \sum_{d\in D}\ind_{\{T=d\}}\mathbb{E}_{\Lambda(y,\cdot)}[g(\mathbf{X}_d)|\mathcal{F}^{\Phi(\mathbf{X})}_d]\\
&=\ind_{\{T<\infty\}} \sum_{d\in D}\ind_{\{T=d\}}\int \Lambda(\mathbf{X}_d,\ddr x)g(x)\\
&=\ind_{\{T<\infty\}}\int \Lambda(\mathbf{X}_T,\ddr x)g(x).
\end{align*}
Setting $T_n:=([2^{n}T]+1)/2^n$ defines a sequence of stopping times taking their values in a countable set and decreasing to $T$. By assumption, $\mathbf{X}$ has càdlàg sample paths, hence $\mathbf{X}_{T_n} \underset{n\rightarrow \infty}{\longrightarrow} \mathbf{X}_{T}$ a.s. on $\{T<\infty\}$. Assume $g$ continuous bounded. Then by combining \cite[Corollary (2.4), Chap. II]{zbMATH02150787} and  the weak continuity of the kernel $y\mapsto \Lambda(y,\cdot)$, one gets:
\[\ind_{\{T<\infty\}}\mathbb{E}_{\Lambda(y,\cdot)}\big[g(\mathbf{X}_T)|\bigcap_{n\geq 1}\mathcal{F}^{\Phi(\mathbf{X})}_{T_n}\big]=\underset{n\rightarrow \infty}{\lim}\ind_{\{T_n<\infty\}}\int \Lambda(\mathbf{X}_{T_n},\ddr x)g(x)=\ind_{\{T<\infty\}}\int \Lambda(\mathbf{X}_T,\ddr x)g(x). \]
Since by assumption $\phi(\mathbf{X})$ is Feller, the usual augmented filtration is right-continuous, see \cite[Proposition 2.10, p. 93]{zbMATH02150787}, and  \[\bigcap_{n\geq 1}\mathcal{F}^{\Phi(\mathbf{X})}_{T_n}=\mathcal{F}^{\Phi(\mathbf{X})}_{T+}=\mathcal{F}^{\Phi(\mathbf{X})}_{T},\]
see e.g. \cite[Exercice 4.17, Chap.I]{zbMATH02150787} and finally \eqref{eq:cond-law-strong}  holds by a standard argument, see e.g. Billingsley's book \cite[Theorem 1.2]{zbMATH01354815}.

\end{proof}

\begin{rem}[Restriction to compactly supported continuous functions]
\label{rem:continuousbounded}
The identity \eqref{inter-append} is an identity between Markov kernels on $\mathcal S'$, or equivalently between operators on the class $b \mathcal S'$ of bounded measurable functions on $S'$. Under the assumption that $S'$ is locally compact and separable, 
it is enough to check that the two operators coincide on 
the class of compactly supported continuous functions
by the uniqueness part of Riesz theorem, see Rudin \cite[Theorem 2.14 p.40]{RUD}.
\end{rem}

\begin{rem}[Submarkovian intertwining] The setting where $(P_t)$ is a sub-Markovian semigroup only can be recasted in the Markovian setting by the adjunction of cemetery points as follows. Precisely, assume  that all the assumptions of Theorem \ref{thmRG} hold except for the fact that $(P_t)$ and consequently also $(Q_t)$ are sub-Markovian.
We extend $\phi$ by requiring $\bar{\phi}:  \bar{S} := S \cup \{ \delta\} \to \bar{S'} := S' \cup \{ \delta\} $ to satisfy $\bar{\phi}(\delta)=\delta'$, and extend $\Lambda$ to $\bar{\Lambda}: \bar{S}' \to \bar{S}$ by requiring $\bar{\Lambda}(\delta', \{\delta\})=1$. The definitions of $P_t$ and $Q_t$ are accordingly modified to accommodate functions defined on $\delta$ and $\delta'$ by setting $\bar{P_t} f(\delta) = f(\delta)$ for $f \in \mathcal{\bar{S}}$ and $\bar{Q_t} f(\delta') = f(\delta')$ for $f \in \mathcal{ \bar{S}'}$. Then one easily checks that the three identities 
$$\bar{\Lambda} \bar{\Phi} =I, \quad  \bar{Q}_t = \bar{\Lambda} \bar{P}_t \bar{\Phi}, \quad \bar{\Lambda} \bar{P}_t = \bar{Q}_t \bar{\Lambda}$$
again hold (for the obvious definition of $\bar{\Phi}$ from $\bar{\phi}$, and as consequence, 
the identity \eqref{eq:cond-law} again holds for $A \in \bar{S}$ with the extended kernel $\bar{\Lambda}$. Last, $\bar{\phi} \circ \bar{X}$ is Markov on $\bar{S'}$ with semi-group $\bar Q_t.$
\end{rem}

\noindent \textbf{Acknowledgements.}   Cl\'ement Foucart is supported  by the European Union (ERC, SINGER, 101054787). Views and opinions expressed are however those of the authors only and do not necessarily reflect those of the European Union or the European Research Council. Neither the European Union nor the granting authority can be held responsible for them.

\end{document}